%% file: main.tex
\documentclass[a4paper, 11pt]{amsart}

\usepackage{comment}
\usepackage{xypic}
\usepackage{graphicx}
\usepackage{tikz-cd}
\input xy
\xyoption{all}

\title[Thurston compactifications of stability manifolds of curves]
{Thurston compactifications of spaces of stability conditions on curves}
\author{Kohei Kikuta, Naoki Koseki, and Genki Ouchi}
\date{}

\address{Department of Mathematics, Graduate School of Science, Osaka University, Toyonaka Osaka, 560-0043, Japan.}
\email{kikuta@math.sci.osaka-u.ac.jp}

\address{The University of Liverpool, Mathematical Sciences Building, Liverpool, L69 7ZL, UK.}
\email{koseki@liverpool.ac.uk}

\address{Graduate School of Mathematics, Nagoya University, Furocho, Chikusaku, Nagoya, 464-8602, Japan.}
\email{genki.ouchi@math.nagoya-u.ac.jp}

\address{}
\email{}

\makeatletter
 
  \@addtoreset{equation}{section}
 \makeatother

\usepackage{amsmath, amssymb, amsthm, amscd, comment, mathtools, color}
\usepackage{todonotes}
\usepackage[frame,cmtip,curve,arrow,matrix,line,graph]{xy}
\usepackage[mathscr]{euscript}
\usepackage{mathrsfs}
\usepackage{tikz}
\usetikzlibrary{intersections, calc}

\input{packages}

\begin{document}
\maketitle

\begin{abstract}
In this paper, we construct a compactification of
the space of Bridgeland stability conditions on a smooth projective curve, as an analogue of Thurston compactifications in Teichm{\"u}ller theory. 

In the case of elliptic curves, 
we compare our results with the classical one of the torus via homological mirror symmetry and give the Nielsen--Thurston classification of autoequivalences using the compactification. 

Furthermore, we observe an interesting phenomenon 
in the case of the projective line. 
\end{abstract}

\setcounter{tocdepth}{1}
\tableofcontents


%
    
%
%


\section{Introduction}
\subsection{Backgrounds}
\subsubsection{Teichm\"{u}ller theory}
Teichm\"{u}ller space $\mcT(\Sigma_g)$ of a closed oriented surface $\Sigma_g$ (of genus $g$) is a fundamental object in the study of moduli spaces of surfaces and mapping class groups. 
Construction of its compactification 
has been a central topic in the field, 
and plays a key role 
in the study of asymptotic properties of hyperbolic and complex geometry, mapping class groups, 
quasi-Fuchsian groups, and so on. 
There are several types of compactifications 
such as the Thurston compactification, the Gardier--Masur compactification, and the Bers compactification. 

\subsubsection{Bridgeland stability conditions}
Bridgeland introduced the notion of stability conditions on triangulated categories as 
a mathematical understanding of 
Douglas' $\Pi$-stability for D-branes in string theory (\cite{bri}). 
In his expository article (\cite{Bri-Teich}) in 2006, 
he pointed out that the space of ``framed" $N=2$ superconformal field theories (called Teichm\"uller space) is related to the space of stability conditions, 
and that the moduli space of $N=2$ superconformal field theories is related to the quotient of the space of stability conditions by the autoequivalence group. 
In another context, motivated by the work of 
Gaiotto--Moore--Neitzke (\cite{GMN}) announced in 2009, Bridgeland--Smith (\cite{BS}) gave a mathematically rigorous proof of the identification between the moduli spaces of meromorphic quadratic differentials on Riemann surfaces 
and the spaces of stability conditions on some CY3 triangulated categories associated with surfaces in 2015. 
These works suggest that there is 
an analogy between Teichm\"uller spaces $\mcT(\Sigma_g)$ (resp. mapping class groups) and the spaces $\Stab(\mcD)$ of stability conditions (resp. autoequivalence groups). 

In this paper, we investigate a compactification 
of the space of Bridgeland stability conditions 
as an analogue of the Thurston compactification 
in Teichm\"uller theory, following Bapat--Deopurkar--Licata (\cite{bdl20}). 

Let us first recall the construction of the classical Thurston compactification. 
It is given by taking the closure of the following embedding associated to the length $l_t(\gamma)$ of a closed geodesic free homotopic to a simple closed curve $\gamma\subset\Sigma_g$: 
\[
\bP l: \mcT(\Sigma_g) \to \bP^\mcS_{\geq 0};~t \mapsto \left[l_t(\gamma)\right]_{\gamma\in\mcS}, 
\]
where $\bP^\mcS_{\geq 0}$ is an infinite dimensional projective space whose homogenous coordinate is given by  $\gamma\in\mcS$. 





Under the above analogy between the Teichm\"uller space and the space of stability conditions, 
the length of a closed geodesic corresponds to the notion of \emph{mass} $m_\sigma(E)\in\bR_{>0}$ of an object $E\in\mcD$ with respect to a stability condition $\sigma\in\Stab(\mcD)$. 
Therefore it is natural to consider 
the following map: 
\[
\bP m: \Stab(\mcD)/\bC \to \bP^\mcS_{\geq 0};~\sigma \mapsto \left[m_\sigma(E)\right]_{E\in\mcS}. 
\]
When the map $\bP m$ is homeomorphic onto the image and the closure is compact, we call it the \emph{Thurston compactification} of $\Stab(\mcD)/\bC$ and denote it by $\overline{\Stab(\mcD)/\bC}$. 
In \cite{bdl20}, Bapat--Deopurkar--Licata constructed Thurston compactifications in this way 
in the case of CY2 categories associated with finite connected quivers, see also \cite{BBL}. 



\subsection{Main results}

The goal of this paper is to study the Thurston compactification of the space of stability conditions in an algebro-geometric setting, 
namely, for the derived category of coherent sheaves on a smooth projective curve. 


Let $C$ be a smooth projective complex curve, $D^b(C)$ the bounded derived category of coherent sheaves on $C$, and $\Stab(C)$ the space of stability conditions on $D^b(C)$. 
A stability condition on $D^b(C)$ 
is called \emph{geometric} 
if all the structure sheaves of points are stable of the same phase. 
We denote the set of geometric stability conditions 
by $\Geo(C)$. 
The first result is the construction of a compactification of $\Geo(C)/\bC$: 
\begin{thm}[Proposition \ref{prop:Geo-inj}, Lemma \ref{lem:defbarm} and Theorem \ref{thm:closure}]
\label{thm:intromain}
The continuous map
\[
\bP m \colon \Geo(C)/\bC \to \bP^\mcS_{\geq 0} 
\]
is homeomorphic onto the image 
and its closure 
$
\overline{\Geo(C)/\bC}
$ 
is homeomorphic to the closed disk. 
In particular, it is compact.  
\end{thm}
Note that the set $\Geo(C)$ of geometric stability conditions is equal to $\Stab(C)$ in the case of positive genus. 
Hence the above theorem gives a Thurston compactification of the whole space $\Stab(C)/\bC$ 
in that case.


\subsubsection{The case of elliptic curves}
In the case of elliptic curves, we compare 
our compactification 
with the classical one for the torus $\Sigma_1$
via the homological mirror symmetry 
\[
\Tilde{\Phi}_{\mathrm{PZ}}:D^b(X) \iso D^\pi\Fuk(\Tilde{X})
\]
due to Polishchuk--Zaslow (\cite{pz98}). 
Here, $X$ is an elliptic curve and $D^\pi\Fuk(\Tilde{X})$ is the derived Fukaya category of the mirror torus $\Tilde{X}$. 
Let $\Sph(X)$ and $\Sph(\Tilde{X})$ be the sets of isomorphism classes of spherical objects in $D^b(X)$ and $D^\pi\Fuk(\Tilde{X})$ respectively.
\begin{thm}[Theorem \ref{thm:main elliptic}]
There exist continuous maps $\eta$ and $\iota$ 
such that the following diagram commutes: 
\begin{center}
\begin{tikzcd}
  \Stab(X)/\bC \ar[r, hookrightarrow, "\bP m"] \arrow[d, "\sim"sloped, "\Tilde{\Phi}_{\mathrm{PZ}}"'] & \bP^{\Sph(X)}_{\geq 0} \ar[d, "\sim"sloped, "\bP\Tilde{\Phi}_{\mathrm{PZ}}"'] \\
  \Stab(\Tilde{X})/\bC \ar[r, hookrightarrow, "\bP m"] & \bP^{\Sph(\Tilde{X})}_{\geq 0}\\
      \mcT(\Sigma_1) \ar[r, hookrightarrow, "\bP l "] \arrow[u, "\eta", "\sim"'sloped] & \bP^{\mcS}_{\geq 0} \arrow[u, "\iota", hookrightarrow].
\end{tikzcd}
\end{center}
Moreover, it induces 
the homeomorphisms between Thurston compactifications: 
\[
\overline{\Stab(X)/\bC} 
\simeq
\overline{\Stab(\Tilde{X})/\bC}
\simeq
\overline{\mcT(\Sigma_1)}. 
\]
\end{thm}
We also obtain a description of the boundary $\partial \Stab(X)/\bC$ of the Thurston compactification $\overline{\Stab(X)/\bC}$ in terms of the intersection numbers defined as the absolute value of the Euler pairing: 
\[i_X:\Sph(X)\times\Sph(X) \to \bZ_{\geq 0}, \quad 
(E,F) \mapsto |\chi(E,F)|,\]
\[i_{\Tilde{X}}:\Sph(\Tilde{X})\times\Sph(\Tilde{X}) \to \bZ_{\geq 0}, \quad 
(E,F) \mapsto |\chi(E,F)|.\]


\begin{thm}[Theorem \ref{thm:main elliptic}]
We have
\[
\partial \Stab(X)/\bC=\overline{i_{X*}(\Sph(X))},
\]
\[
\partial \Stab(\Tilde{X})/\bC=\overline{i_{\Tilde{X}*}(\Sph(\Tilde{X}))},
\]
where $i_{X*}\colon \Sph(X) \to \bP^{\Sph(X)}_{\geq 0}$ 
and $i_{\Tilde{X}*}\colon \Sph(\Tilde{X}) \to \bP^{\Sph(\Tilde{X})}_{\geq 0}$ are maps induced by the intersection numbers {\rm (Definition \ref{def-intersection})}. 
\end{thm}

Moreover, we give the characterizations of 
the so-called 
{\it periodic, reducible}, and {\it pseudo-Anosov} 
autoequivalces of $D^b(X)$ 
using the Thurston compactification 
(Proposition \ref{prop-ell}, \ref{prop-para} and \ref{prop-hyp}). 
It yields the Nielsen--Thurston classification of autoequivalences: 
\begin{thm}[Nielsen--Thurston classification, Theorem \ref{thm-NT}]
Each autoequivalence of $D^b(X)$, 
which acts on $\Stab(X)/\bC$ non-trivially, 
is of exactly one of the following types: periodic, reducible, or pseudo-Anosov. 
\end{thm}

\subsubsection{The case of the projective line}
In the case of the projective line $\bP^1$, 
the map $\bP m$ from the whole space 
$\Stab(\bP^1)/\bC(\supsetneq{\Geo(\bP^1)/\bC})$
is unfortunately not injective:  
\begin{prop}[Proposition \ref{non-injectivity-p^1}]
The map 
\[
\bP m \colon \Stab(\bP^1)/\bC \to \bP^{\mcS} 
\]
is NOT injective 
for any choice of a set 
$\mcS \subset \Ob(D^b(\bP^1))$. 
\end{prop}
A non-injectivity of the map $\bP m$ 
is related to the existence of 
a semiorthogonal decomposition 
(Remark \ref{rem:non-inj}). 
We however capture an interesting phenomenon: 
the image $\bP m(\Stab(\bP^1)/\bC)$ gives a partial compactification of $\Geo(\bP^1)/\bC$ 
(Theorem \ref{thm-explicit-im-p^1} and Figure 2). 

\vspace{2mm}
The case of K3 surfaces is also considered in our ongoing work (\cite{KKOk3}). 

\subsection{Outline of the paper}
The paper is organized as follows: 
In Section \ref{sec:prel}, we review basic definitions in the theory of Bridgeland stability conditions. 
In Section \ref{section-thurston-cpt}, we quickly review Teichm{\"u}ller theory and propose some problems regarding the constructions of Thurston compactifications for the spaces of stability conditions. 
In Section \ref{sec:Geo}, we prove Theorem \ref{thm:intromain}. 
In Sections \ref{section-ell} and \ref{sec:P1}, 
we investigate the cases of elliptic curves and the projective line in further details. 

\begin{ACK}
The authors would like to thank Professor Arend Bayer 
for valuable discussions. 
K.K. is supported by JSPS KAKENHI Grant Number 20K22310 and 21K13780. 
N.K. was supported by ERC Consolidator grant WallCrossAG, no.~819864. 
G.O. is supported by JSPS KAKENHI Grant Number 19K14520. 
\end{ACK}

\begin{NaC}
Throughout the paper, we work over the complex number field $\bC$. 
For a smooth projective variety $X$, $D^b(X)$ denotes the bounded derived category of coherent sheaves on $X$.
\end{NaC}

\section{Preliminaries on Bridgeland stability conditions} \label{sec:prel}

In this section, we recall the definition of stability conditions and the basic properties of the space of stability conditions. 

Throughout this paper, $\mcD$ is a triangulated category of finite type over $\mathbb{C}$, and $\Aut(\mcD)$ is the group of autoequivalences of $\mcD$. We denote the Grothendieck group of $\mcD$ by $K(\mcD)$. 

Fix a finitely generated free abelian group $\Lambda$, a surjective group homomorphism $\cl: K(\mcD)\twoheadrightarrow\Lambda$ and a norm $||\cdot||$ on $\Lambda\otimes_\bZ \bR$. 
We moreover assume the existence of a group homomorpshim $\alpha: \Aut(\mcD)\to\Aut_{\bZ}(\Lambda)$ such that the following diagram is commutative for all $\Phi \in\Aut(\mcD)$: 
\[
\xymatrix{
    K(\mcD) \ar[r]^{K(\Phi)} \ar[d]_{\cl} & K(\mcD) \ar[d]_{\cl} \\
    \Lambda \ar[r]^{\alpha(\Phi)} & \Lambda. 
}
\]
If such $\alpha$ exists, it is uniquely determined. 

\begin{defin}[{\cite[Definition 5.1]{bri}}]
A {\it stability condition} $\sigma=(Z,\mcP)$ on $\mcD$ (with respect to $(\Lambda, \cl)$) consists of a group homomorphism
$Z: \Lambda\to\bC$ called a {\it central charge} and a family $\mcP=\{\mcP(\phi)\}_{\phi\in\bR}$ of full additive subcategory of $\mcD$ called a {\it slicing}, such that
\begin{enumerate}
\item
For $\phi \in \mathbb{R}$ and $0\neq E\in\mcP(\phi)$, we have $Z(\cl(E))=m(E)\exp(i\pi\phi)$ for some $m(E)\in\bR_{>0}$. 
\item
For all $\phi\in\bR$, we have $\mcP(\phi+1)=\mcP(\phi)[1]$.
\item
For $\phi_1>\phi_2$ and $E_i\in\mcP(\phi_i)$, we have $\Hom(E_1,E_2)= 0$.
\item
For each $0\neq E\in\mcD$, there is a collection of exact triangles called {\it Harder--Narasimhan filtration} of $E$: 
\begin{equation}\label{HN}
\begin{xy}
(0,5) *{0=E_0}="0", (20,5)*{E_{1}}="1", (30,5)*{\dots}, (40,5)*{E_{p-1}}="k-1", (60,5)*{E_p=E}="k",
(10,-5)*{A_{1}}="n1", (30,-5)*{\dots}, (50,-5)*{A_{p}}="nk",
\ar "0"; "1"
\ar "1"; "n1"
\ar@{.>} "n1";"0"
\ar "k-1"; "k" 
\ar "k"; "nk"
\ar@{.>} "nk";"k-1"
\end{xy}
\end{equation}
with $A_i\in\mcP(\phi_i)$ and $\phi_1>\phi_2>\cdots>\phi_p$. 
\item
(support property)
There exists a constant $C>0$ such that for all $\phi \in \mathbb{R}$ and $0\neq E\in\mcP(\phi)$, 
we have 
\begin{equation}
||\cl(E)||<C|Z(\cl(E))|.
\end{equation}
\end{enumerate}
\end{defin}
For any interval $I\subset\bR$, define $\mcP(I)$ to be the extension-closed subcategory of $\mcD$ generated by the subcategories $\mcP(\phi)$ for $\phi\in I$. 
Then $\mcP((0,1])$ is the heart of a bounded t-structure on $\mcD$, hence an abelian category.  
The full subcategory $\mcP(\phi)\subset\mcD$ is also shown to be abelian. 
A non-zero object $E\in\mcP(\phi)$ is called a {\it $\sigma$-semistable object} of {\it phase} $\phi_\sigma(E):=\phi$, and a simple object in $\mcP(\phi)$ is called a {\it $\sigma$-stable object}. 
Taking the Harder--Narasimhan filtration (\ref{HN}) of $E \in \mcD$, we define $\phi^+_\sigma(E):=\phi_\sigma(A_1)$ and $\phi^-_\sigma(E):=\phi_\sigma(A_p)$. 
The object $A_i$ is called {\it $\sigma$-semistable factor} of $E$. 
Define ${\rm Stab}_\Lambda(\mcD)$ to be the set of stability conditions on $\mcD$ with respect to $(\Lambda,\cl)$.

In this paper, we assume that the space ${\rm Stab}_\Lambda(\mcD)$ is not empty. 
By abuse notation, we write $Z(E)$ instead of $Z(\cl(E))$. 

We prepare some terminologies on stability functions on the heart of a $t$-structure on $\mcD$. 
\begin{defin}
Let $\mcA$ be the heart of a bounded $t$-structure on $\mcD$. 
A {\it stability function} on $\mcA$ is a group homomorphism $Z: \Lambda\to\bC$ such that for all $0\neq E\in\mcA\subset\mcD$, the complex numbers $Z(E)$ lie in the semiclosed upper half plane $\bH_{-}:=\{re^{i\pi\phi}\in\bC~|~r\in\bR_{>0},\phi\in(0,1]\}\subset\bC$. 
\end{defin}
Given a stability function $Z: \Lambda\to\bC$ on $\mcA$, the {\it phase} of an object $0\neq E\in\mcA$ is defined to be $\phi(E):=\frac{1}{\pi}{\rm arg}Z(E)\in(0,1]$. 
An object $0\neq E\in\mcA$ is {\it $Z$-semistable} (resp. {\it $Z$-stable}) if for all subobjects $0\neq A\subset E$, we have $\phi(A)\le\phi(E)$ (resp. $\phi(A)<\phi(E)$). 
We say that a stability function $Z$ satisfies {\it the Harder--Narasimhan property} if 
each object $0\neq E\in\mcA$ admits a filtration (called the Harder--Narasimhan filtration of $E$) 
$0=E_0\subset E_1\subset E_2\subset\cdots\subset E_m=E$ such that $E_i/E_{i-1}$ is $Z$-semistable for $i=1,\cdots,m$ with $\phi(E_1/E_0)>\phi(E_2/E_1)>\cdots>\phi(E_m/E_{m-1})$. 
A stability function $Z$ satisfies 
{\it the support property} if there exists a constant $C>0$ such that for all $Z$-semistable objects $E\in\mcA$, 
we have $||\cl(E)||<C|Z(\cl(E))|$. 

The following proposition shows the relationship between stability conditions and stability functions on the heart of a bounded $t$-structure. 
\begin{prop}[{\cite[Proposition 5.3]{bri}}]
Giving a stability condition on $\mcD$ is equivalent to giving the heart $\mcA$ of a bounded t-structure on $\mcD$, and a stability function $Z$ on $\mcA$ with the Harder--Narasimhan property and the support property. 
\end{prop}
For the proof, we construct the slicing $\mcP$ from the pair $(Z,\mcA)$ by 
\[
\mcP(\phi):=\{E\in\mcA~|~E\text{ is }Z\text{-semistable with }\phi(E)=\phi\} \cup \{ 0\}\text{ for }\phi\in(0,1],
\]
and extend it for all $\phi\in\bR$ by $\mcP(\phi+1):=\mcP(\phi)[1]$. 
Conversely, for a stability condition $\sigma=(Z,\mcP)$, the heart $\mcA$ is given by $\mcA:=\mcP_\sigma((0,1])$. 
We also denote a stability condition as $(Z,\mcA)$. 

The following notion is important to analyze the space of stability conditions: 
\begin{defin}
Let $E\in\mcD$ be a non-zero object of $\mcD$ and $\sigma\in {\rm Stab}_\Lambda(\mcD)$ be a stability condition on $\mcD$. 
The {\it mass} $m_{\sigma}(E)\in\bR_{>0}$ of $E$ is defined by
\begin{equation}
m_{\sigma}(E):=\displaystyle\sum_{i=1}^p|Z_\sigma(A_i)|, 
\end{equation}
where $A_1,\cdots,A_p$ are $\sigma$-semistable factors of $E$. 
\end{defin}

The following generalized metric (i.e. with values in $[0,\infty]$) is defined by Bridgeland. 
\begin{defin}[{\cite[Proposition 8.1]{bri}}]
The generalized metric $d_B$ on ${\rm Stab}_\Lambda(\mcD)$ is defined by 
\begin{equation}
d_B(\sigma, \tau):=\sup_{E\neq0}\left\{|\phi^+_\sigma(E)-\phi^+_\tau(E)|, |\phi^-_\sigma(E)-\phi^-_\tau(E)|, \middle|\log\frac{m_\sigma(E)}{m_\tau(E)}\middle|\right\}\in[0,\infty]. 
\end{equation}
\end{defin}
This generalized metric induces the topology on ${\rm Stab}_\Lambda(\mcD)$, 
and it takes finite values on each connected component ${\rm Stab}_\Lambda^\dagger(\mcD)$ of ${\rm Stab}_\Lambda(\mcD)$. 
Thus $({\rm Stab}_\Lambda^\dagger(\mcD),d_B)$ is a metric space in the strict sense. 
\begin{thm}[{\cite[Theorem 7.1]{bri}}]
The map
\begin{equation}
{\rm Stab}_\Lambda(\mcD)\to\Hom_\bZ(\Lambda, \bC);~\sigma=(Z,\mcP)\mapsto Z
\end{equation}
is a local homeomorphism, where $\Hom_\bZ(\Lambda, \bC)$ is equipped with the natural linear topology. 
\end{thm}
Therefore the space ${\rm Stab}_\Lambda(\mcD)$ (and each connected component ${\rm Stab}_\Lambda^\dagger(\mcD)$) naturally admits a structure of a finite dimensional complex manifold. 
\begin{thm}[{\cite[Thorem 3.6]{Woo}}]
The metric space $({\rm Stab}_\Lambda^{\dagger}(\mcD), d_B)$ is complete. 
Moreover, the limit point $\sigma_\infty=(Z_\infty, \mcP_\infty)$ of a Cauchy sequence $\{\sigma_n=(Z_n,\mcP_n)\}_n\subset {\rm Stab}_\Lambda^\dagger(\mcD)$ is described by
\begin{eqnarray*}
&&Z_\infty=\lim_{n\to\infty}Z_n,\\
&&\mcP_\infty(\phi)=\langle0\neq E\in\mcT~|~\lim_{n\to\infty}\phi^+_{\sigma_n}(E)=\phi,~\lim_{n\to\infty}\phi^-_{\sigma_n}(E)=\phi\rangle\cup\{0\}. 
\end{eqnarray*}
\end{thm}
There are two group-actions on ${\rm Stab}_\Lambda(\mcD)$. 
The first one is a left $\Aut(\mcD)$-action defined by 
\begin{equation}
\Phi\sigma:=(Z_\sigma(\alpha(\Phi^{-1})\cl(-)), \{\Phi(\mcP_\sigma(\phi))\})~\text{ for }\sigma\in {\rm Stab}_\Lambda(\mcD),~\Phi\in\Aut(\mcD). 
\end{equation}
Let $\widetilde{{\rm GL}}_+(2,\bR)\to{\rm GL}_+(2,\bR)$ be the universal cover of ${\rm GL}_+(2,\bR)$. 
Then $\widetilde{{\rm GL}}_+(2,\bR)$ is isomorphic to the group of pairs $(M,f)$, where $f:\bR\to\bR$ is an increasing map with $f(\phi+1)=f(\phi)+1$, and $M\in{\rm GL}_+(2,\bR)$ such that the induced maps on $(\bR^2\backslash\{0\})/\bR_{>0}=S^1=\bR/2\bZ$ coincide. 
We have a right $\widetilde{{\rm GL}}_+(2,\bR)$-action
on $\Stab_\Lambda(\mcD)$ defined by
\begin{equation}\label{univ-cov-gl}
\sigma.g:=(M^{-1}\circ Z_\sigma, \{\mcP_\sigma(f(\phi))\})~\text{ for }\sigma\in {\rm Stab}_\Lambda(\mcD),~g=(M,f)\in\widetilde{{\rm GL}}_+(2,\bR). 
\end{equation}
These two actions commute, and the $\widetilde{{\rm GL}}_+(2,\bR)$-action is free and continuous, 
hence preserves any connected component of ${\rm Stab}_\Lambda(\mcD)$. 
The restriction of the $\widetilde{{\rm GL}}_+(2,\bR)$-action to the subgroup $\bC\subset\widetilde{{\rm GL}}_+(2,\bR)$ is given as follows: 
\begin{equation}
\sigma.\lambda=(\exp(-\sqrt{-1}\pi\lambda)\cdot Z_\sigma, \{\mcP_\sigma(\phi+{\rm Re}\hspace{0.5mm}\lambda)\})\text{ for }\lambda\in\bC. 
\end{equation}
\begin{lem}
The $\Aut(\mcD)$-action and the $\bC$-action are isometric with respect to $d_B$. 
\end{lem}
Let $\Stab_\Lambda^\dagger(\mcD)$ be a connected component of $\Stab_\Lambda(\mcD)$.
The quotient metric of $d_B$ by the $\bC$-action is well-defined and compatible with the quotient topology on $\Stab_\Lambda^\dagger(\mcD)/\bC$ (see \cite[\S3.1]{Kik-curvature}). 
Hence $\Stab_\Lambda^\dagger(\mcD)/\bC$ is metrizable. 
As in \cite[Proposition 4.1]{Oka}, $\Stab_\Lambda^\dagger(\mcD)/\bC$ also admits a structure of a complex manifold. 
Therefore, $\Stab_\Lambda^\dagger(\mcD)/\bC$ is second-countable, paracompact and $\sigma$-compact. 



\section{Thurston compactifications}\label{section-thurston-cpt}

In this section, we discuss Thurston compactifications of Teichm\"uller spaces and spaces of stability conditions.

\subsection{Teichm\"uller spaces} \label{sec:Thurston}
Fix a positive integer $g$. Let $\Sigma_g$ be a closed oriented surface of genus $g$.
A {\it marked Riemann surface of genus} $g$ is an orientation preserving diffeomorphism $f: \Sigma_g \to X$, where $X$ is a compact Riemann surface. Two marked Riemann surfaces $f_1: \Sigma_g \to X_1$ and $f_2: \Sigma_g \to X_2$ are said to be {\it equivalent} if there is a bi-holomorphic map $h: X_1 \iso X_2$ such that $f_2 \circ f^{-1}_1$ and $h$ are isotopic. The {\it Teichm\"uller space} $\mcT( \Sigma_g)$ consists of equivalence classes of marked Riemann surfaces of genus $g$.
In this paper, we will mainly use the case of $g=1$.

\begin{ex}\label{example:Teichmuler elliptic}
Assume that $g=1$. 
Fix the torus $\Sigma_1=T_2:=\bR^2/\bZ^2$.
For $\tau \in \bH$, we consider the orientation preserving diffeomorphism 
\[f_\tau: T_2 \to \bC/\bZ \oplus \tau \bZ, ~ (x,y) \mapsto x+\tau y. \]
Then we have the isomorphism 
\[ \xi:\bH \iso \mcT(T_2),~ \tau \mapsto [f_\tau].\]
\end{ex}

To construct the Thurston compactification of $\mcT(\Sigma_g)$, we need to consider infinite dimensional projective spaces.

\begin{defin}\label{def:proj}
For a set $\mcS$, we define the {\it projective space} $\bP^{\mcS}_{\geq 0}$ as the topological space
\[\bP^\mcS_{\geq 0}:=(\bR_{\geq 0}^\mcS\setminus\{0\})/\bR_{>0},\]
where the topology of $\bP^\mcS_{\geq 0}$ is defined by 
the quotient topology of the product topology on $\bR^{\mcS}_{\geq 0} \setminus \{0\}$. 
For a point $(x_\gamma)_{\gamma \in {\mcS}} \in \bR^S_{\geq 0}\setminus \{0\}$, denote the corresponding point in $\bP^\mcS_{\geq 0}$ by $[x_\gamma]_{\gamma \in \mcS}$.
Note that the topological space $\bP_{\geq 0}^\mcS$ is 
Hausdorff 
and  path-connected. 
If $\mcS$ is a countable set, then $\bP_{\geq 0}^\mcS$ is second-countable. 


\end{defin}
In the situation of Teichm{\"u}ller theory, we choose a set $\mcS$ as the set of free homotopy classes of simple closed curves in $\Sigma_g$. 
We explain the construction of a map 
$l \colon \mcT(\Sigma_g) \to \bP^\mcS_{\geq 0}$. 
Take $\gamma \in \mcS$ and $t=[f:\Sigma_g \to X] \in \mcT(\Sigma_g)$. 
When $g=1$, there is a geodesic in $X$ free homotopic to $f_*\gamma$, which is unique up to translations. When $g\geq 2$, there is a unique geodesic in $X$ free homotopic to $f_*\gamma$.
The length $l_t(\gamma)$ of a geodesic in $X$ free homotopic to $f_*\gamma$ depends only on the classes $t$ and $\gamma$. 
For an element $t \in \mcT(\Sigma_g)$, we put $l(t):=(l_t(\gamma))_{\gamma \in \mcS}$.
Then we obtain the continuous maps \[l: \mcT(\Sigma_g) \to \bR^S_{\geq 0}\setminus \{0\},\]  \[\bP l: \mcT(\Sigma_g) \to \bP^\mcS_{\geq 0}.\]

\begin{ex}[{\cite[Subsection 1.2]{FLPV}}]\label{ex:line segment}
Let $\gamma_1, \gamma_2 \in H_1(T_2,\bZ)$ be homology classes defined by 
    \[\gamma_1:=\left[\pi\left(\{(0,y) \in \bR^2 \mid y \in \bR  \}\right)\right], \]
    \[\gamma_2:=\left[\pi\left(\{(x,0) \in \bR^2 \mid x \in \bR \}\right)\right].\]
    Then $\gamma_1, \gamma_2$ form a symplectic basis of $H_1(T_2,\bZ)$, which satisfies  
    \[\gamma^2_1=\gamma^2_2=0, ~ \gamma_2 \cdot \gamma_1=1.\]
    There is the natural one-to-one correspondence between $\mcS$ and the set 
    \[\{c_1 \gamma_1 +  c_2 \gamma_2 \in H_1(T_2,\bZ) \mid \gcd(c_1, c_2)=1 \}\]
    and we identify these sets. 
    For $c_1 \gamma_1+c_2 \gamma_2 \in \mcS$, $\tau \in \bH$ and $t:=[f_\tau]$,
    we have 
    \begin{equation}\label{eq:length T_2}
    l_{t}(c_1\gamma_1+c_2\gamma_2)=|c_2+c_1\tau|. 
    \end{equation}
    Take a geodesic $L$ in $\bC/\bZ \oplus \tau \bZ$ whose free homotopy class is $c_1\gamma_1+c_2\gamma_2$. Take the lift $\Tilde{L}$ of $L$ in the universal cover $\bR^2$ of $T_2$. 
    Then the right hand side in (\ref{eq:length T_2}) is just the length of the line segment $\Tilde{L}$.
    \end{ex}

Let $\MCG(\Sigma_g)$ be the mapping class group of $\Sigma_g$.
For a mapping class $[\varphi] \in \MCG(\Sigma_g)$ and an element $t=[f:\Sigma_g \to X] \in \mcT(\Sigma_g)$, we put
\[[\varphi]\cdot t:=[f \circ \varphi^{-1}:\Sigma_g \to X] \in \mcT(\Sigma_g). \]
It defines an action of the mapping class group $\MCG(\Sigma_g)$ on the Teichm\"uller space $\mcT(\Sigma_g)$.
Note that the mapping class group $\MCG(\Sigma_g)$ naturally acts on the set $\mcS$.
For a mapping class $[\varphi] \in \MCG(\Sigma_g)$ and $x=(x_\gamma)_{\gamma \in S}$, we put \[[\varphi]\cdot x:= x \circ [\varphi]^{-1},  \]
where we regard $x$ and $[\varphi]^{-1}_*$ as the maps $x:S \to \bR^\mcS_{\geq 0}\setminus \{0\}$ and $[\varphi]^{-1}_*:\mcS \to \mcS$ respectively. Then the maps $l$ and $\bP l$ are $\MCG(\Sigma_g)$-equivariant.

\begin{thm}[{\cite[Theorem 1.2]{FLPV}}]\label{thm:Thurston comactification}
The following statements hold.
\begin{itemize}
    \item[(1)] The map $\bP l: \mcT(\Sigma_g) \to \bP^{S}_{\geq 0}$ is a homeomorphism onto its image.
    \item[(2)] Let $\overline{\mcT(\Sigma_g)}$ be the closure of $\bP l(\mcT(\Sigma_g))$ in $\bP^{S}_{\geq 0}$. Then $\overline{\mcT(\Sigma_g)}$ is compact. When $g=1$, $\overline{\mcT(\Sigma_g)}$ is homeomorphic to $\bH \cup \bP^1_{\bR}$. When $g \geq 2$, $\overline{\mcT(\Sigma_g)}$ is homeomorphic to the closed ball of dimension $6g-6$.
    \item[(3)] For $\gamma_1, \gamma_2 \in \mcS$, we define the geometric intersection number $i(\gamma_1, \gamma_2)$ of $\gamma_1$ and $\gamma_2$ as the infimum of the number of intersections of simple closed curves $L_1$ and $L_2$ whose free homotopy classes are $\gamma_1$ and $\gamma_2$ respectively. For $\gamma \in \mcS$, we define the function $i_*(\gamma):\mcS \to \bZ_{\geq 0}$ by $i_*(\gamma)(\delta):=i(\gamma, \delta)$ for $\delta \in \mcS$. Then we have 
    \[ \partial\mcT(\Sigma_g)=\overline{i_*(\mcS)}.  \]
    \item[(4)] The action of the mapping class group $\MCG(\Sigma_g)$ on $\mcT(\Sigma_g)$ is naturally extended to the action on $\overline{\mcT(\Sigma_g)}$ continuously.

\end{itemize}
\end{thm}

In the case of $g=1$, the function $i(-,-): \mcS \times \mcS \to \bZ_{\geq 0}$ can be described as follows:

\begin{ex}[{\cite[Subsection 1.2]{FLPV}}]
Assume that $g=1$ and $\Sigma_1=T_2$.
For $a\gamma_1+b\gamma_2, c\gamma_1+d\gamma_2 \in \mcS$, we have 
\[i(a\gamma_1+b\gamma_2, c\gamma_1+d\gamma_2)=|bc-ad|.\]
\end{ex}

\subsection{Spaces of stability conditions}
In this subsection, we introduce problems about Thurston compactifications of spaces of stability conditions following \cite{bdl20}.
We use the same setting as Subsection 2.1.
Let $\mcS$ be a subset of the set of isomorphism classes of objects in $\mcD$. By Definition \ref{def:proj}, we obtain the projective space $\bP^{\mcS}_{\geq0}$. 

\begin{defin}[{\cite[Subsection 1.1]{bdl20}}]\label{
def-pm}
We define the continuous map 
\[m:\Stab_{\Lambda}(\mcD) \to \bR^{\mcS}_{\geq 0}\setminus \{0\}, \ \sigma \mapsto (m_\sigma(E))_{E \in \mcS}. \]
This map induces the continuous map 
\[
\bP m:\Stab_{\Lambda}(\mcD)/\bC \to \bP^{\mcS}_{\geq 0}.
\]
\end{defin}

As an analogue of Theorem \ref{thm:Thurston comactification}, we consider the following questions.

\begin{ques}[{\cite[Section 1]{bdl20}}]\label{ques:Thurston compactification}
\begin{itemize}
    \item[(0)] Is the map 
    \[\bP m: \Stab_{\Lambda}(\mcD)/\bC \to \bP^{\mcS}_{\geq 0}\]
    injective?
    \item[(1)]Is the map 
    \[\bP m: \Stab_{\Lambda}(\mcD)/\bC \to \bP^{\mcS}_{\geq 0}\]
    a homeomorphism onto its image?
    \item[(2)] Assume that the question (1) has an affirmative answer. 
    Denote  the closure of $\bP m(\Stab_{\Lambda}(\mcD)/\bC)$ in $\bP^{\mcS}_{\geq 0}$ by $\overline{\Stab_{\Lambda}(\mcD)/\bC}$. 
    Is the closure $\overline{\Stab_{\Lambda}(\mcD)/\bC}$ compact?
    \item[(3)] Assume that the questions (1) and (2) have affirmative answers. 
    Set $\partial\Stab_{\Lambda}(\mcD)/\bC:=\left(\overline{\Stab_{\Lambda}(\mcD)/\bC}\right)\backslash\bP m\left(\Stab_{\Lambda}(\mcD)/\bC\right)$. 
    Is there a map $i:\mcS \times \mcS \to \bR_{\geq 0}$ such that 
    \[\partial\Stab_{\Lambda}(\mcD)/\bC=\overline{i_*(\mcS)}?\]    Here, for a map $i:\mcS \times \mcS \to \bR_{\geq 0}$ and an element $A \in \mcS$, we define the map 
    \[i_*(A):\mcS \to \bR_{\geq 0}\]
    by $i_*(A)(E):=i(A,E)$ for $E \in \mcS$. 
\end{itemize}
\end{ques}

Assume that $\Aut(\mcD)$ preserves the set $\mcS$.
For an autoequivalence $\Phi \in \Aut(\mcD)$ and $x=(x_E)_{E \in S}$, we put \[\Phi\cdot x:= x \circ \Phi^{-1},  \]
where we regard $x$ and $\Phi^{-1}_*$ as the maps $x:\mcS \to \bR^\mcS_{\geq 0}\setminus \{0\}$ and $\Phi^{-1}_*:\mcS \to \mcS$ respectively. Then the maps $m$ and $\bP m$ are $\Aut(\mcD)$-equivariant.

\begin{rmk}
Assume that Question \ref{ques:Thurston compactification} (1), (2) and (3) have affirmative answers. Then the action of $\Aut(\mcD)$ on $\Stab_{\Lambda}(\mcD)/\bC$ is naturally extended to the action of $\Aut(\mcD)$ on $\overline{\Stab_{\Lambda}(\mcD)/\bC}$ continuously.
\end{rmk}

\section{Thurston compactification of 
the space of geometric stability conditions}
\label{sec:Geo}
Let $C$ be a smooth projective curve. 
In this section, we construct Thurston compactification of the space of geometric stability conditions on $C$.
We also describe the image explicitly. 

\subsection{Basics}
Let $C$ be a smooth projective curve over $\bC$.
Consider the surjective homomorphism $\ch: K(C) \to H^{2*}(C,\bZ)$. Denote the space of stability conditions on $D^b(C)$ with respect to $(H^{2*}(C,\bZ), \ch)$ by $\Stab(C)$.  
A stability condition $\sigma \in \Stab(C)$ is {\it geometric} if all structure sheaves of points are $\sigma$-stable of the same phase. 
The set of geometric stability conditions is denoted by $\Geo(C)$. 

For $\beta+\sqrt{-1}\alpha\in\bH$, 
one defines 
a group homomorphism $Z_{\beta, \alpha} (\text{or } Z_{\beta+\sqrt{-1}\alpha}):H^{2*}(C,\bZ) \to\bC$ by
\[
Z_{\beta, \alpha}(r,d):=
-d+(\beta+\sqrt{-1}\alpha)r. 
\]
Then we have $\sigma_{\beta,\alpha}:=(Z_{\beta,\alpha},\Coh(C)) \in \Geo(C)$.
By \cite[Theorerm 2.7]{Mac} and arguments in \cite[Section 3]{BMW},
there exists an isomorphism
\begin{equation}\label{isomorphism:H vs Stab}
    \bH\times\bC\xrightarrow{\sim}\Geo(C);~~(\beta+\sqrt{-1}\alpha,\lambda)\mapsto\sigma_{\beta,\alpha}.\lambda. 
\end{equation}

The following is obtained by the isomorphism (\ref{isomorphism:H vs Stab}):  
\begin{lem}\label{curv-stab-determined-by-central}
Let $C$ be a smooth projective curve. 
For $\sigma_1,\sigma_2\in\Geo(C)$, 
$\overline{\sigma}_1=\overline{\sigma}_2$ in $\Geo(C)/\bC$ if and only if 
$Z_{\sigma_1}=\exp(-\sqrt{-1}\pi\lambda) Z_{\sigma_2}$ for some $\lambda\in\bC$. 
\end{lem}

In the case of positive genus, each stability condition is geometric. 
\begin{thm}[{\cite[Theorerm 2.7]{Mac}}]
Let $C$ be a smooth projective curve of positive genus. 
Then we have $\Stab(C)=\Geo(C)$. 
\end{thm}

\vspace{1mm}
For the projective line $\bP^1$, the subset $\Geo(\bP^1)$ is strictly contained in $\Stab(\bP^1)$, see \S \ref{subsection-exp-im-p1}. 

\subsection{Compactifications}
We fix a set $\mcS \subset \Ob\left(D^b(C) \right)$ 
satisfying the following conditions: 
\begin{assump}\label{assume:S}
The set $\mcS \subset \Ob\left(D^b(C) \right)$ satisfies the following conditions:   
\begin{enumerate}
    \item The structure sheaf $\mcO_C$ is in $\mcS$.
    \item There are points $p, q \in C$ such that 
    $\mcO_p, \mcO_C(-q) \in \mcS$. 
     \item Every object in $\mcS$ is slope stable. 
\end{enumerate}
\end{assump} \label{eq:Ssph}

First, we prove the injectivity of the map 
$\bP m \colon \Geo(C)/\bC \to \bP^\mcS_{\geq 0}$ in Definition \ref{
def-pm}: 
\begin{prop} \label{prop:Geo-inj}
The map 
\begin{equation} \label{eq:mGeo}
\bP m \colon \Geo(C)/\bC \to \bP^\mcS_{\geq 0} 
\end{equation}
is injective. 
\end{prop}
\begin{proof}
By the isomorphism (\ref{isomorphism:H vs Stab}), we have the isomorphism
\begin{equation} \label{eq:H-Geo}
\bH \xrightarrow{\sim} \Geo(C)/\bC, 
\quad \beta+\sqrt{-1}\alpha \mapsto 
\overline{\sigma}_{\beta, \alpha}.
\end{equation}
Take points 
$\beta+\sqrt{-1}\alpha, \beta'+\sqrt{-1}\alpha' \in \bH$ and assume that 
$\bP m(\overline{\sigma}_{\beta, \alpha})
=\bP m(\overline{\sigma}_{\beta', \alpha'})$. 
To prove the injectivity of $\bP m$, it is enough to prove that $\beta+\sqrt{-1}\alpha=\beta'+\sqrt{-1}\alpha'$. 

By Assumption \ref{assume:S} (1), (2), 
we have $\mcO_C, \mcO_C(-q), \mcO_p \in \mcS$ for some 
points $p, q \in C$. 
Since the equality
\[m_{\sigma_{\beta,\alpha}}(\mcO_{p})=1=m_{\sigma_{\beta', \alpha'}}(\mcO_p) \]
holds,
we have $m_{\sigma_{\beta,\alpha}}(E)=m_{\sigma_{\beta',\alpha'}(E)}$ for all $E \in \mcS$.
In particular, we have
\[
m_{\sigma_{\beta, \alpha}}\left(\mcO_C(-q) \right)
=m_{\sigma_{\beta', \alpha'}}\left(\mcO_C(-q) \right), \quad 
m_{\sigma_{\beta, \alpha}}\left(\mcO_C \right)
=m_{\sigma_{\beta', \alpha'}}\left(\mcO_C \right).
\]
Therefore, we obtain 
\[
\alpha^2+(\beta+1)^2=\alpha'^2+(\beta'+1)^2, \quad 
\alpha^2+\beta^2=\alpha'^2+\beta'^2. 
\]
By solving these equations, we get 
$\beta+\sqrt{-1}\alpha=\beta'+\sqrt{-1}\alpha'$ 
as required. 



\end{proof}


In the following, 
we describe the closure of 
$\bP m(\Geo(C)/\bC) \subset \bP^{\mcS}_{\geq 0}$. 
Let $D$ (resp. $\overline{D}$) 
be an open (resp. a closed) unit disc, 
and consider the homeomorphism 
\begin{equation} \label{eq:homeo-H}
\bH \xrightarrow{\sim} D, 
\quad z \mapsto \frac{z-\sqrt{-1}}{z+\sqrt{-1}}. 
\end{equation}
The above map (\ref{eq:homeo-H}) extends to the homeomorphism 
\begin{equation} \label{eq:homeo-bH}
    \overline{\bH} \xrightarrow{\sim} \overline{D}, 
\end{equation}
where we put 
$\overline{\bH} \coloneqq \bH \cup \bR \cup \{\infty\}$. 

\begin{lem} \label{lem:defbarm}
The map (\ref{eq:mGeo}) extends to a continuous map 
\begin{equation} \label{eq:barm}
\overline{\bP m} \colon \overline{\bH} \to \bP^{\mcS}_{\geq 0}. 
\end{equation} 
Moreover, it is injective and homeomorphic onto the image. 
\end{lem}
\begin{proof}
Under the identification (\ref{eq:H-Geo}), 
the map (\ref{eq:mGeo}) becomes 
\begin{equation} \label{eq:mH-Geo}
\bH \to \bP^{\mcS}_{\geq 0}, \quad 
\beta+\sqrt{-1}\alpha \mapsto 
\left[|Z_{\beta, \alpha}(E)| \right]_{E \in \mcS}. 
\end{equation}
Note that we have used Assumption \ref{assume:S} (3) 
for the equality 
$m_{\sigma_{\beta, \alpha}}(E)=|Z_{\beta, \alpha}(E)|$ 
for $E \in \mcS$. 
Explicitly, we have 
\[
|Z_{\beta, \alpha}(E)|=\left| 
-\deg(E)+(\beta+\sqrt{-1}\alpha)\rk(E)
\right|. 
\]
This map (\ref{eq:mH-Geo}) naturally extends to a continuous map 
\begin{equation} 
\overline{\bP m} \colon \overline{\bH} \to \bP^{\mcS}_{\geq 0}
\end{equation}
by sending $\beta+\sqrt{-1}\alpha \in \bH \cup \bR$ to 
$[\left| 
-\deg(E)+(\beta+\sqrt{-1}\alpha)\rk(E)
\right|]_{E \in \mcS} \in \bP^{\mcS}_{\geq 0}$, 
and $\infty \in \overline{\bH}$ to 
$[\rk(E)]_{E \in \mcS}$. 

By the same proof as in Proposition \ref{prop:Geo-inj}, 
we can see that the morphism $\overline{{\bP m}}$ is injective. 
Since $\overline{\bH}$ is compact and 
$\bP^{\mcS}_{\geq 0}$ is Hausdorff, 
$\overline{{\bP m}}$ is a homeomorphism onto its image. 
\end{proof}

\begin{thm} \label{thm:closure}
The closure of $\bP m(\Geo(C)/\bC) \subset \bP^\mcS_{\geq 0}$ is homeomorphic to the closed unit disc $\overline{D}$. 
In particular, it is compact. 
\end{thm}
\begin{proof}
As before, we identify $\Geo(C)/\bC$ with $\bH$ 
by (\ref{eq:H-Geo}). 
Since $\overline{\bP m}$ is continuous, we have 
\[
\bP m(\bH) \subset \overline{\bP m}(\overline{\bH}) \subset 
\overline{\bP m(\bH)}. 
\]
Moreover, since $\overline{{\bP m}}(\overline{\bH})$ is compact 
and $\bP^\mcS_{\geq 0}$ is Hausdorff, 
$\overline{\bP m}(\overline{\bH}) \subset \bP^\mcS_{\geq 0}$ 
is closed. 
We conclude that 
\[
\overline{D} \simeq \overline{\bP m}(\overline{\bH}) = 
\overline{\bP m(\bH)}, 
\]
where the first isomorphism follows from Lemma \ref{lem:defbarm}. 
\end{proof}

By Proposition \ref{prop:Geo-inj} and Theorem \ref{thm:closure}, we obtain affirmative answers to Question \ref{ques:Thurston compactification} (1) and (2) for any smooth projective curve $C$.
\subsection{Explicit description of the image of $\bP m$}
Here, We describe the image of $\bP m$ explicitly. 

For a point $p\in C$, we set $\mcS_3:=\{\mcO_p,\mcO_C(-p),\mcO_C\}$ and
put $\bR\bP^2_{\geq 0}:=\bP^{\mcS_3}_{\geq 0}$. Consider the map
\[
\bP m \colon \Geo(C)/\bC \to \bR\bP^2_{\geq 0}. 
\]
By  Lemma \ref{lem:defbarm} and Theorem \ref{thm:closure}, $\bP m$ is a homeomorphism onto its image and the closure of the image is compact. 
We define a subset $\Delta \subset\mathbb{R}\bP^2_{\ge0}$ as follows: 
\[
\Delta:=\{[1:X:Y]\in\bR\bP^2_{\ge0}\mid~Y>-X+1,~Y>X-1,~Y<X+1\}
\]
We obtain the following result: 
\begin{prop}\label{geom-chamber-explicit}
The image $\bP m(\Geo(C)/\bC)$ is equal to $\Delta$. 
\end{prop}
\begin{proof}
Take $\beta+\sqrt{-1}\alpha \in \bH$.
Applying the triangle inequality of mass (\cite[Proposition 3.3]{Ike}) for the exact sequence 
\[
0\to\mcO_C(-p)\to\mcO_C\to\mcO_p\to0,
\]
one has 
\begin{eqnarray*}
1&\le& m_{\sigma_{\beta,\alpha}}(\mcO_C(-p))+m_{\sigma_{\beta,\alpha}}(\mcO_C),\\
m_{\sigma_{\beta,\alpha}}(\mcO_C(-p))&\le& m_{\sigma_{\beta,\alpha}}(\mcO_C)+1,\\
m_{\sigma_{\beta,\alpha}}(\mcO_C)&\le& m_{\sigma_{\beta,\alpha}}(\mcO_C(-p))+1. 
\end{eqnarray*}
All these inequalities are strict since all objects in $\mcS_3$ are $\sigma_{\beta,\alpha}$-stable and the phase of $\mcO_p$ is greater than that of the others. 
Alternatively, we can also check the inequalities by direct calculations. 
Therefore we have an inclusion $\bP m(\Geo(C)/\bC)\subset\Delta$. 

We now consider the reverse inclusion. 
For any $[1:X:Y]\in\Delta$, solving the equations
\[
\begin{cases}
X=\sqrt{(\beta+1)^2+\alpha^2}\\
Y=\sqrt{\beta^2+\alpha^2},
\end{cases}
\]
we have
\[
\begin{cases}
\beta=\frac{1}{2}(X^2-Y^2-1)\\
\alpha=\frac{1}{2}\sqrt{(X+Y-1)(X-Y+1)(X+Y+1)(-X+Y+1)}. 
\end{cases}
.\]
Note that $\alpha$ is a positive real number since $[1:X:Y] \in \Delta$.
Therefore, we have $\bP m(\overline{\sigma}_{\beta, \alpha})=[1:X:Y]$.
\end{proof}

\begin{figure}[htbp]
\begin{center}
\begin{tikzpicture}[scale=1]
\draw [->] (-1, 0) -- (3, 0) node[right]{$X$}; 
\draw[->] (0, -1) -- (0, 4) node[right]{$Y$}; 
\draw[dashed, domain=-1:2] plot(\x, -\x+1);  
\draw (-1.3, 2) node[above]{$Y=-X+1$}; 
\draw[dashed, domain=0:3] plot(\x, \x-1) node[right]{$Y=X-1$}; 
\draw[dashed, domain=-1:3] plot(\x, \x+1) node[right]{$Y=X+1$}; 
\fill[lightgray] (1, 0)--(0, 1)--(3, 4)--(3, 2)--cycle;
\end{tikzpicture}
\end{center}
\caption{The image $\bP m(\Geo(C)/\bC)$ in $\bR\bP^2_{\geq 0}$.}
\end{figure}

\newpage

\section{The case of elliptic curves}\label{section-ell}




The case of elliptic curves is discussed. 
We compare the Thurston compactification with the classical one of the torus via homological mirror symmetry in the first two subsections, 
and give the Nielsen--Thurston classification of autoequivalences in the rest. 

For an elliptic curve $X$, recall that we have 
\[
\Stab(X)=\Geo(X). 
\]

\subsection{Homological mirror symmetry for elliptic curves}
In this subsection, we recall the homological mirror symmetry for elliptic curves following \cite{pz98}. 
For $\zeta\in \bH$, 
we consider the elliptic curve $X:=\bC/\bZ\oplus \zeta\bZ$ and a pair $\Tilde{X}:=(T_2,\zeta~dx \wedge dy)$, where $T_2=\bR^2/\bZ^2$ is the torus. Polishchuk and Zaslow \cite{pz98} constructed the equivalence 
\[\Phi_{\mathrm{PZ}}:D^b(X) \iso D^\pi\Fuk(\Tilde{X}).\]
Let $\pi: \bR^2 \to T_2$ be the natural projection.
Recall that an indecomposable object in the derived Fukaya category $D^\pi\Fuk(\Tilde{X})$ of $\Tilde{X}$ is isomorphic to a triple $(L, \lambda, M)$, where $L$ is a Lagrangian submanifold of $\Tilde{X}$ and $\lambda$ is a real number such that  
 \[L=\pi\left(\{ 
    z \in \bC \mid z=z_0+e^{i\pi\lambda}t, 
    ~ t \in \bR
    \}\right),\]
    and $M$ is a local system on $L$ whose monodoromy operators have only eigenvalues in the unitary group $U(1)$. 
    
    We define a surjective homomorphism 
    \[\cl:K(D^\pi\Fuk(\Tilde{X})) \to H_1(T_2,\bZ)\]
    as $\cl(L,\lambda,M):=[L]$ for $(L,\lambda,M) \in D^\pi\Fuk(\Tilde{X})$ with $\lambda \in (-1/2,1/2]$, 
    and extend it for a general element in 
    $D^\pi\Fuk(\tilde{X})$ 
    by using $(L,\lambda,M)[1]=(L,\lambda+1,M)$.
   For $(r_1,d_1),(r_2,d_2) \in H^{2*}(X,\bZ)$, we define the {\it Mukai pairing
 of $(r_1, d_1)$ and $(r_2,d_2)$} as 
 \[\langle(r_1,d_1), (r_2,d_2) \rangle:=r_1d_2-r_2d_1.\]
 For $E_1,E_2 \in D^b(X)$, we have 
\begin{equation}\label{eq:Riemann-Roch}
\chi(E_1,E_2)=\langle \ch(E_1), \ch(E_2)\rangle 
 \end{equation}
    by the Riemann--Roch formula. 
 From now on, we use the same notation as in Example \ref{ex:line segment}. 
We have an isomorphism 
    \[\varphi_{\mathrm{PZ}}:H^{2*}(X,\bZ) \iso H_1(T_2, \bZ), (r,d) \mapsto r \gamma_2+ d \gamma_1.\]
Note that $\varphi_{\mathrm{PZ}}$ is the isometry with respect to the Mukai pairing and the intersection pairing. Moreover, we have the following commutative diagram:

\begin{center}
\begin{tikzcd}
  K(X) \ar[r, "K(\Phi_{\mathrm{PZ}})"] \arrow[d, "\ch"'] & K(D^\pi\Fuk(\Tilde{X})) \ar[d, "\cl"'] \\
  H^{2*}(X,\bZ) \ar[r, "\varphi_{\mathrm{PZ}}"] & H_1(T_2,\bZ).
\end{tikzcd}
\end{center}

For an indecomposable coherent sheaf $E$ on $X$, the object $\Phi_{\mathrm{PZ}}(E)$ is isomorphic to a triple $(L,\lambda,M)$ such that
\[[L]=\rk (E) \gamma_2+ \deg (E) \gamma_1\] 
and $\lambda \in (-1/2,1/2]$. 

Let $\mcP \in D^b(X\times X)$ be the normalized Poincar\'e line bundle of $X$. By \cite{muk81}, the Fourier-Mukai transform 
\[\Phi_\mcP: D^b(X) \to D^b(X),~ E \mapsto \mathbf{R}p_{1*}(p^*_2E \otimes \mcP)\]
is an autoequivalence, where $p_1$ and $p_2$ are the first projection and the second projection respectively. Then we have the cohomological Fourier-Mukai transform 
\begin{equation}\label{coh-Phi_P}
\Phi^H_\mcP:H^{2*}(X,\bZ) \iso H^{2*}(X,\bZ),~(r,d) \mapsto (d,-r), 
\end{equation}
which is an isometry with respect to the Mukai pairing.

We consider the equivalence 
\[\Tilde{\Phi}_{\mathrm{PZ}}:=\Phi_{\mathrm{PZ}} \circ \Phi^{-1}_{\mcP}:D^b(X) \iso D^\pi\Fuk(\Tilde{X})\]
and the isomorphism 
\[\Tilde{\varphi}_{\mathrm{PZ}}:=\varphi_{\mathrm{PZ}} \circ (\Phi^{-1}_{\mcP})^H: H^{2*}(X,\bZ) \iso H_1(T_2,\bZ).\]
For $(r,d) \in H^{2*}(X,\bZ),$ we have $\Tilde{\varphi}_{\mathrm{PZ}}(r,d)=-d\gamma_2+r\gamma_1$.

\begin{rmk}\label{rmk:isometry}
The isomorphism $\Tilde{\varphi}_{\mathrm{PZ}}$ is isometry with respect to the Mukai pairing on $H^{2*}(X,\bZ)$ and the intersection pairing on $H_1(T_2,\bZ)$.
\end{rmk}

\begin{defin}\label{def:stability on Fukaya}
For a complex number $\beta+\sqrt{-1}\alpha \in \bH$, we define a stability condition $\Tilde{\sigma}_{\beta,\alpha}$ 
on $D^\pi\Fuk(\Tilde{X})$ with respect to $H_1(T_2, \bZ)$ 
as follows: 
\[\Tilde{\sigma}_{\beta,\alpha}:=(Z_{\beta, \alpha} \circ \Tilde{\varphi}^{-1}_{\mathrm{PZ}},\Tilde{\Phi}_{\mathrm{PZ}}(\Coh(X))).\]
\end{defin}

\begin{rmk}\label{rmk:period integral}
Central charges of the stability conditions in Definition \ref{def:stability on Fukaya} are related to period integrals.
Take a complex number $\tau=\beta+\sqrt{-1} \alpha \in \bH$. 
Then we have the orientation preserving diffeomorphism $f_\tau: T_2 \to  \bC/\bZ \oplus \tau\bZ$ in Example \ref{example:Teichmuler elliptic}. There is the unique holomorphic $1$-form $\Omega_\tau$ on the elliptic curve $\bC/\bZ \oplus \tau\bZ$ such that 
\[\int_{f_{\tau*}\gamma_1}\Omega_\tau=\tau, ~\int_{f_{\tau*}\gamma_2}\Omega_\tau=1. \]
Then we have the equality
\[\Tilde{Z}_{\beta, \alpha}(-)=\int_{(-)}\Tilde{\Omega}_\tau, \]
where $\Tilde{\Omega}_\tau:=f^*_\tau \Omega_\tau$.
In fact, we can prove this equality as follows.
First, note that the equality
\[\int_{\gamma_k}\Tilde{\Omega}_\tau=\int_{f_{\tau*}\gamma_k}\Omega_\tau\]
holds for $k=1,2$.
Take a class $\gamma=c_1 \gamma_1  + c_2 \gamma_2 \in H_1(T_2,\bZ)$, where $c_1$ and $c_2$ are integers. 
Then we obtain
\begin{align*}
    \Tilde{Z}_{\beta,\alpha}(\gamma) &= Z_{\beta, \alpha}(\Phi^H_\mcP \circ \varphi^{-1}_{\mathrm{PZ}}(\gamma)) \\
    &=Z_{\beta, \alpha}(c_1, -c_2)\\
    &=c_2+\tau c_1\\
    &=c_2 \int_{\gamma_2}\Tilde{\Omega}_\tau + c_1\int_{\gamma_1}\Tilde{\Omega}_\tau\\
    &= \int_{\gamma}\Tilde{\Omega}_\tau.
\end{align*}
\end{rmk}


\subsection{Comparison of two Thurston compactifications}
In this subsection, we compare Thurston compactifications of spaces of stability conditions with Thurston compactifications of Teichm\"uller spaces in the case of elliptic curves.
We keep the notations as in the previous subsection. 

Recall that an object $E \in D^b(X)$ is called {\it spherical} if we have
\[ \mathbf{R}\Hom(E,E) \simeq \bC \oplus \bC[-1].\]

\begin{defin}
Let $\Sph(X)$ and $\Sph(\Tilde{X})$ be the sets of isomorphism classes of spherical objects in $D^b(X)$ and $D^\pi\Fuk(\Tilde{X})$ respectively.
\end{defin}

Spherical objects on elliptic curves are characterized as follows:

\begin{prop}\label{prop:spherical}
Let $\sigma \in \Stab(X)$ be a stability condition on $D^b(X)$.
Then a non-zero object $E \in D^b(X)$ is 
spherical if and only if it is $\sigma$-stable.
In particular, if $E$ is spherical, then we have $\gcd(\rk(E), \deg(E))=1$.
\end{prop}
\begin{proof}
Assume that $E$ is spherical. 
Since $\Hom(E,E)=\bC$, the object $E$ is indecomposable. By the proof of \cite[Theorem 8.1]{bri}, $E$ is a $\sigma$-semistable object. By \cite[Theorem 8.1]{bri} and taking shifts, we may assume that $E$ is a $\mu$-semistable sheaf on $X$. By \cite[Lemma 1, Proposition 4]{hp} and the equality $\Hom(E,E)=\bC$, $E$ is $\mu$-stable.
Therefore, $E$ is $\sigma$-stable.
The converse is easily deduced from the Serre duality.
\end{proof}

Now, we obtain the following proposition.
\begin{prop}\label{prop:mass=length}
Take a complex number $\tau \in \bH$ and put $t:=[f_\tau] \in \mcT(T_2)$. 
For a spherical object $(L, \lambda, M) \in D^\pi\Fuk(\Tilde{X})$, we have 
\[m_{\Tilde{\sigma}_{\beta,\alpha}}(L,\lambda,M)=l_{t}([L]). \]
\end{prop}
\begin{proof}
By Proposition \ref{prop:spherical}, $(L, \lambda, M)$ is $\Tilde{\sigma}_{\beta,\alpha}$-stable.
Therefore, by example \ref{ex:line segment} and Remark \ref{rmk:period integral}, we obtain
\begin{align*}
m_{\Tilde{\sigma}_{\beta,\alpha}}(L,\lambda,M)&=\left|\Tilde{Z}_{\beta,\alpha}([L])\right|\\
&=\left| \int_{\gamma}\Tilde{\Omega}_\tau \right|\\
&=|c_2+\tau c_1|\\
&=l_t([L]).
\end{align*}
\end{proof}

Denote the space of stability conditions on $D^\pi\Fuk(\Tilde{X})$ with respect to $\cl$ by $\Stab(\Tilde{X})$.
The equivalence $\Tilde{\Phi}_{\mathrm{PZ}}:D^b(X) \iso D^\pi\Fuk(\Tilde{X})$ induces isomorphisms 
\[\Tilde{\Phi}_{\mathrm{PZ}}: \Stab(X)/\bC \iso \Stab(\Tilde{X})/\bC,\]
\[ \bP\left(\Tilde{\Phi}_{\mathrm{PZ}}\right): \bP^{\Sph(X)}_{\geq 0} \iso \bP^{\Sph(\Tilde{X})}_{\geq 0}, \]
\begin{rmk}\label{diag-mirror}
The above isomorphisms fit into the following commutative diagram: 
\begin{center}
\begin{tikzcd}
  \Stab(X)/\bC \ar[r, "\bP m"] \arrow[d, "\Tilde{\Phi}_{\mathrm{PZ}}"'] & \bP^{\Sph(X)}_{\geq 0} \ar[d, "\bP\Tilde{\Phi}_{\mathrm{PZ}}"'] \\
  \Stab(\Tilde{X})/\bC \ar[r, "\bP m"] & \bP^{\Sph(\Tilde{X})}_{\geq 0}.
\end{tikzcd}
\end{center}
\end{rmk}

Let $\mcS$ be the set of free homotopy classes of simple closed curves in $T_2$ as in Example \ref{ex:line segment}. 

\begin{rmk}\label{rmk:spherical class}
For $(L,\lambda, M) \in \Sph(\tilde{X})$, 
there are integers $c_1$ and $c_2$ such that $[L]=c_1 \gamma_1+c_2 \gamma_2$ and $\gcd(c_1,c_2)=1$ by the equivalence $\Tilde{\Phi}_{\mathrm{PZ}}$ and Proposition \ref{prop:spherical}.
\end{rmk}

By Proposition \ref{prop:spherical} and Remark \ref{rmk:spherical class}, 
we can define the following map.

\begin{defin}
Take a point $x=(x_{\gamma})_{\gamma \in \mcS} \in \bP^{\mcS}_{\geq 0}$. 
For a spherical object $(L,\lambda,M) \in \Sph(\Tilde{X})$, we define 
$y_{(L,\lambda,M)}:=x_{[L]}$. Defining
\[\iota(x):=(y_{(L,\lambda,M)})_{(L,\lambda,M)\in \Sph(\Tilde{X})}, \]
we obtain an injective continuous map
\[\iota:\bP^{\mcS}_{\geq 0} \to \bP^{\Sph(\Tilde{X})}_{\geq 0}. \]
\end{defin}

By the isomorphism (\ref{isomorphism:H vs Stab}), we have the isomorphism 
\[\varepsilon : \bH \iso \Stab(X)/\bC. \]
We consider the isomorphism 
\[\eta:=\Tilde{\Phi}_{\mathrm{PZ}} \circ \varepsilon \circ \xi^{-1}: \mcT(T_2) \iso \Stab(\Tilde{X})/\bC,\]
where $\xi:\bH \iso \mcT(T_2)$ is the isomorphism in Example \ref{example:Teichmuler elliptic}. For $\tau=\beta+\sqrt{-1}\alpha \in \bH$, we have $\eta([f_\tau])=\Tilde{\sigma}_{\beta,\alpha}$. 

\begin{defin}\label{def-intersection}
We define functions $i_X$ and $i_{\Tilde{X}}$ as follows: 
\[i_X:\Sph(X)\times\Sph(X) \to \bZ_{\geq 0}, \quad 
(E,F) \mapsto |\chi(E,F)|,\]
\[i_{\Tilde{X}}:\Sph(\Tilde{X})\times\Sph(\Tilde{X}) \to \bZ_{\geq 0}, \quad 
(E,F) \mapsto |\chi(E,F)|.\]
Then we define maps $i_{X*}$ and $i_{\Tilde{X}*}$ 
as in Question \ref{ques:Thurston compactification} (3): 
\begin{equation*}
i_{X*}:\Sph(X) \to \bP^{\Sph(X)}_{\geq 0}, \quad
i_{\Tilde{X}*}:\Sph(\Tilde{X}) \to \bP^{\Sph(\Tilde{X})}_{\geq 0}. 
\end{equation*}
\end{defin}

Recall that
\begin{eqnarray*}
\overline{\Stab(X)/\bC}&:=&\overline{\bP m(\Stab(X)/\bC)}\\
\partial\Stab(X)/\bC&:=&\left(\overline{\Stab(X)/\bC}\right)\backslash\bP m\left(\Stab(X)/\bC\right), 
\end{eqnarray*}
cf. Question \ref{ques:Thurston compactification} (2),(3). 
The following is the main theorem of this section. 
\begin{thm}\label{thm:main elliptic}
The following diagram commutes:
\begin{center}
\begin{tikzcd}
  \Stab(X)/\bC \ar[r, hookrightarrow, "\bP m"] \arrow[d, "\sim"sloped, "\Tilde{\Phi}_{\mathrm{PZ}}"'] & \bP^{\Sph(X)}_{\geq 0} \ar[d, "\sim"sloped, "\bP\Tilde{\Phi}_{\mathrm{PZ}}"'] \\
  \Stab(\Tilde{X})/\bC \ar[r, hookrightarrow, "\bP m"] & \bP^{\Sph(\Tilde{X})}_{\geq 0}\\
      \mcT(T_2) \ar[r, hookrightarrow, "\bP l "] \arrow[u, "\eta", "\sim"'sloped] & \bP^{\mcS}_{\geq 0} \arrow[u, "\iota", hookrightarrow].
\end{tikzcd}
\end{center}
Furthermore, the following statements hold.
\begin{itemize}
    \item[(1)] The maps given by the restrictions
    \[\bP \Tilde{\Phi}_{\mathrm{PZ}}: \overline{\Stab(X)/\bC} \to \overline{\Stab(\Tilde{X})/\bC},
\]
\[\iota:\overline{\mcT(T_2)} \to \overline{\Stab(\Tilde{X})/\bC}, 
\]
are homeomorphisms.
\item[(2)]
For $E \in \Sph(X)$ with $\Tilde{\Phi}_{\mathrm{PZ}}(E)=(L,\lambda,M)$, we have 
\[\bP\Tilde{\Phi}_{\mathrm{PZ}} (i_{X*}(E))=i_{\Tilde{X}*}(L,\lambda,M)=\iota(i_{*}(\cl(L,\lambda,M))). \]

In particular, we obtain 
\[\partial \Stab(X)/\bC=\overline{i_{X*}(\Sph(X))},\]
\[\partial \Stab(\Tilde{X})/\bC=\overline{i_{\Tilde{X}*}(\Sph(\Tilde{X}))}.\]

\end{itemize}
\end{thm}
\begin{proof}
The commutativity of the diagrams is deduced from Remark \ref{diag-mirror} and Proposition \ref{prop:mass=length}. 

We prove the statement (1).
Since $\iota$ is continuous, we have $\iota\left(\overline{\mcT(T_2)}\right) \subset \overline{\iota(\bP l(\mcT(T_2)))}=\overline{\Stab(\Tilde{X})/\bC}$. Therefore, we obtain the injective continuous map 
$\iota: \overline{\mcT(T_2)} \to \overline{\Stab(\Tilde{X})/\bC}$.
Since $\overline{\mcT(T_2)}$ is compact and $\overline{\Stab(\Tilde{X})/\bC}$ is Hausdorff, $\iota: \overline{\mcT(T_2)} \to \overline{\Stab(\Tilde{X})/\bC}$ is a homeomorphism onto its image. Since $\iota\left(\overline{\mcT(T_2)}\right)$ is a compact subspace of the Hausdorff space  $\overline{\Stab(\Tilde{X})/\bC}$, $\iota\left(\overline{\mcT(T_2)}\right)$ is closed. 
By the inclusion $\bP m(\Stab(\Tilde{X})/\bC) \subset \iota\left(\overline{\mcT(T_2)}\right)$, we obtain \[  \iota\left(\overline{\mcT(T_2)}\right)=\overline{\Stab(\Tilde{X})/\bC}.\]

Next, we prove the statement (2). 
Take $(L',\lambda',M') \in \Sph(\Tilde{X})$. Then we have 
\begin{align*}
    \bP \Tilde{\Phi}_{\mathrm{PZ}}(i_{X*}(E))(L',\lambda',M')&=i_{X*}(E)(\Tilde{\Phi}^{-1}_{\mathrm{PZ}}(L',\lambda',M')) \\
    &=i_{X}(E,\Tilde{\Phi}^{-1}_{\mathrm{PZ}}(L',\lambda',M'))\\
    &=i_{\Tilde{X}}((L,\lambda,M),(L',\lambda',M'))\\
    &=i_{\Tilde{X}*}(L, \lambda, M)(L',\lambda',M')
\end{align*}
by (\ref{eq:Riemann-Roch}) and Remark \ref{rmk:isometry}. and 
\begin{align*}
    \iota(i_{*}(\cl(L,\lambda,M)))(L',\lambda',M')&=i([L],[L'])\\
    &=i_{\Tilde{X}*}(L,\lambda,M)(L',\lambda',M')
\end{align*}
by Example \ref{ex:line segment}.
\end{proof}


By Theorem \ref{thm:main elliptic}, we obtain an affirmative answer to Question \ref{ques:Thurston compactification} (3) for an elliptic curve $X$.

\subsection{Nielsen--Thurston classification of autoequivalences and the categorical entropy}
Here, we consider an analogue of the classical Nielsen--Thurston classification of the torus
for autoequivalences of elliptic curves, and its relation to the categorical entropy. 

\subsubsection{${\rm PSL}(2,\bZ)$-actions}

Taking the cohomological Fourier--Mukai transforms, 
we have an exact sequence
\begin{equation}\label{ell-ex}
1 \to \left(\Aut(X) \ltimes \Pic^0(X) \right) \times \bZ[2] \to \Aut(D^b(X)) \to \SL(2,\bZ) \to 1.
\end{equation}


The exact sequence (\ref{ell-ex}) induces an isomorphism 
\[
\rho:~\Gamma(X):=\Aut(D^b(X))/\left((\Aut(X) \ltimes \Pic^0(X))\times \bZ[1]\right)
\xrightarrow{\sim}
{\rm  PSL}(2,\bZ)
.
\]
Since the group $\mcI(X):=\left(\Aut(X) \ltimes \Pic^0(X) \right) \times \bZ[1]$ acts on $\Stab(X)/\bC$ trivially, 
$\Gamma(X)$ naturally acts by isometries on $\Stab(X)/\bC$ with respect to the quotient metric $\bar{d}_B$ of $d_B$ by the $\bC$-action. 

We consider an action of ${\rm  PSL}(2,\bZ)$ on $\bH$
as a variant of M\"{o}bius transformation given by 
\begin{equation}\label{var-Mobius}
A.z:=
\left(
\begin{pmatrix}
0 & 1 \\
1 & 0 \\
\end{pmatrix}
A
\begin{pmatrix}
0 & 1 \\
1 & 0 \\
\end{pmatrix}
\right)
._Mz
\end{equation}
for $A\in{\rm  PSL}(2,\bZ)$ and $z\in\bH$, 
where the action $._M$ on the right hand side is the usual M\"{o}bius transformation
\[
\begin{pmatrix}
a & b \\
c & d \\
\end{pmatrix}
._Mz:=
\frac{az+b}{cz+d}
.
\]
\begin{rmk}\label{rem-var-Mobius}
\begin{enumerate}
\item
Clearly, the action (\ref{var-Mobius}) 
of ${\rm  PSL}(2,\bZ)$ is also isometric with respect to the hyperbolic metric $d_H$ on $\bH$. 
\item
Since the trace is invariant under taking the conjugation, 
for (a representative of) $A\in{\rm  PSL}(2,\bZ)$, we have 
\[
|{\rm tr}A|
=
\left |{\rm tr}
\begin{pmatrix}
0 & 1 \\
1 & 0 \\
\end{pmatrix}
A
\begin{pmatrix}
0 & 1 \\
1 & 0 \\
\end{pmatrix}
\right |.
\]

\end{enumerate}
\end{rmk}
\begin{prop}\label{psl-equiv}
Via the identification $\rho^{-1}:{\rm  PSL}(2,\bZ)\xrightarrow{\sim}\Gamma(X)$,   
the homeomorphism
\[
\varphi:(\bH, d_H) \xrightarrow{\sim} (\Stab(C)/\bC, \bar{d}_B)
\]
is ${\rm PSL}(2,\bZ)$-equivariant and isometric (up to a multiplicative constant). 
\end{prop}
\begin{proof}
By the formula (\ref{coh-Phi_P}), we have 
\[
\rho(\Phi_\mcP)
=
\begin{pmatrix}
0 & 1 \\
-1 & 0 \\
\end{pmatrix}
\text{ and }
\rho(-\otimes\mcO_C(p))
=
\begin{pmatrix}
1 & 0 \\
1 & 1 \\
\end{pmatrix}
.
\]
Hence $\Phi_\mcP$ and $-\otimes\mcO_C(p)$ generate $\Gamma(X)(\simeq{\rm  PSL}(2,\bZ))$. 
It suffices to check that these generators satisfy the equivariance condition. 

For each $\beta+\sqrt{-1}\alpha\in\bH$, we set 
\[
\tau:=
\varphi
\left(
\begin{pmatrix}
0 & 1 \\
-1 & 0 \\
\end{pmatrix}
.(\beta+\sqrt{-1}\alpha)
\right)
\in\Stab(X)/\bC. 
\]
Then direct computations show that 
\begin{eqnarray*}
Z_\tau&=&Z_{\frac{-1}{\beta+\sqrt{-1}\alpha}}\\
&=&\frac{1}{\beta+\sqrt{-1}\alpha}Z_{\beta+\sqrt{-1}\alpha}\circ\Phi_\mcP^{-1}\\
&=&\frac{1}{\beta+\sqrt{-1}\alpha}Z_{(\Phi_\mcP.\sigma_{\beta+\sqrt{-1}\alpha})}. 
\end{eqnarray*}
By Lemma \ref{curv-stab-determined-by-central}, we have 
\[
\varphi
\left(
\begin{pmatrix}
0 & 1 \\
-1 & 0 \\
\end{pmatrix}
.(\beta+\sqrt{-1}\alpha)
\right)
=
\begin{pmatrix}
0 & 1 \\
-1 & 0 \\
\end{pmatrix}
.
\varphi(\beta+\sqrt{-1}\alpha).
\]
Similarly, we can check that
\[
\varphi
\left(
\begin{pmatrix}
1 & 0 \\
1 & 1 \\
\end{pmatrix}
.(\beta+\sqrt{-1}\alpha)
\right)
=
\begin{pmatrix}
1 & 0 \\
1 & 1 \\
\end{pmatrix}
.
\varphi(\beta+\sqrt{-1}\alpha). 
\]
Thus $\varphi$ is ${\rm PSL}(2,\bZ)$-equivariant. 

Let 
\[
\varphi_-:(\bH, d_H) \xrightarrow{\sim} (\bH, d_H);~\beta+\sqrt{-1}\alpha\mapsto -\beta+\sqrt{-1}\alpha
\]
be an orientation reversing isometry, and 
\[
\varphi_W:(\bH, d_H) \xrightarrow{\sim} (\Stab(C)/\bC, \bar{d}_B);~\beta+\sqrt{-1}\alpha\mapsto
\overline{\sigma_{0,1}.
\widetilde{
\begin{pmatrix}
\alpha & \beta \\
0 & 1 \\
\end{pmatrix}
}
}
\]
an identification, 
where
$\widetilde{
\begin{pmatrix}
\alpha & \beta \\
0 & 1 \\
\end{pmatrix}
}$
is a lift of the matrix 
$
\begin{pmatrix}
\alpha & \beta \\
0 & 1 \\
\end{pmatrix}
$
via the universal covering
$
\widetilde{{\rm GL}}_+(2,\bR)\to{\rm GL}_+(2,\bR).
$
Using Lemma \ref{curv-stab-determined-by-central}, 
it is easy to check that 
\[
\overline{\sigma_{0,1}.
\widetilde{
\begin{pmatrix}
\alpha & \beta \\
0 & 1 \\
\end{pmatrix}
}
}
=\overline{\sigma_{-\beta,\alpha}}. 
\]
By Woolf's computation (\cite[Proposition 4.1]{Woo}), $\varphi_W$ is isometric (up to 
scaling by $1/2$). 
Therefore the map $\varphi=\varphi_W\circ\varphi_-$ is also isometric, which completes the proof. 
\end{proof}
Since ${\rm  PSL}(2,\bZ)$ is isomorphic to the isometry group of $(\bH, d_H)$, the following is a direct corollary of the above proposition. 
\begin{cor}\label{non-trivial}
Let $\Phi$ be an autoequivalence of $D^b(X)$. 
Then, $\Phi$ is non-trivial in $\Gamma(X)$ if and only if 
$\Phi$ acts on $\Stab(X)/\bC$ non-trivially. 
\end{cor}
By Lemma \ref{lem:defbarm} and the proof of Theorem \ref{thm:closure}, 
we can extend the isometry $\varphi$ 
to the isomorphism 
$\overline{\bH} \xrightarrow{\sim} \overline{\Stab(X)/\bC}$. 
Also, the action of $\Gamma(X)$ can be extended to its action by homeomorphisms on the compactification $\overline{\Stab(X)/\bC}$ by Proposition \ref{psl-equiv}. 

\subsubsection{Nielsen--Thurston classification}
The {\it translation length} $l_B(\Phi)\in\bR_{\ge0}$ of $\Phi\in\Aut(D^b(X))$
with respect to the isometric action on $(\Stab(X)/\bC,\bar{d}_B)$ is given by
\[
l_B(\Phi):=\displaystyle\inf_{\sigma\in \Stab(X)/\bC}\bar{d}_B(\sigma,\Phi.\sigma). 
\]
\begin{defin}
Let $\Phi\in\Aut(D^b(X))$ be an autoequivalence. 
\begin{enumerate}
\item
$\Phi$ is {\it elliptic} if 
there exists $\sigma\in\Stab(X)/\bC$ such that $l_B(\Phi)=\bar{d}_B(\sigma,\Phi.\sigma)$ and $l_B(\Phi)=0$, or equivalently $\Phi$ has a fixed point in $\Stab(X)/\bC$. 
\item
$\Phi$ is {\it parabolic} if $l_B(\Phi)$ does not attain its minimum. 
\item
$\Phi$ is {\it hyperbolic} if 
there exists $\sigma\in\Stab(X)/\bC$ such that $l_B(\Phi)=\bar{d}_B(\sigma,\Phi.\sigma)$ and $l_B(g)>0$.
\end{enumerate}
\end{defin}
We give the characterizations of the above trichotomy of isometries. 
\begin{prop}\label{prop-ell}
Let $\Phi\in\Aut(D^b(X))$ be an autoequivalence which is non-trivial in $\Gamma(X)$. 
The following conditions are equivalent. 
\begin{enumerate}
\item
$\Phi$ is elliptic. 
\item
$|{\rm tr}\rho(\Phi)|<2$. 
\item
$\Phi$ has a unique fixed point in $\Stab(X)/\bC$. 
\item
(periodic) There exists a positive integer $m\in\bZ_{>0}$ such that $\Phi^m\in\mcI(X)$. 
\end{enumerate}
\end{prop}
\begin{proof}
By Remark \ref {rem-var-Mobius} (2) and Proposition \ref{psl-equiv}, 
the conditions (1), (2) and (3) are all equivalent by the classical facts on M\"{o}bius transformation, see \cite[Section 13.1]{FM}. 

These three conditions are also equivalent to the finiteness of the order of $\Phi$ in $\Gamma(X)$(\cite[Section 13.1]{FM}), i.e., to the condition (4). 
\end{proof}

\begin{prop}\label{prop-para}
Let $\Phi\in\Aut(D^b(X))$ be an autoequivalence which is non-trivial in $\Gamma(X)$. 
The following conditions are equivalent. 
\begin{enumerate}
\item
$\Phi$ is parabolic. 
\item
$|{\rm tr}\rho(\Phi)|=2$. 
\item
$\Phi$ has a unique fixed point in $\partial(\Stab(X)/\bC)$. 
\item
(reducible) There exists a spherical object $E\in D^b(X)$ and $\Psi\in\mcI(X)$ such that $\Phi(E)=\Psi(E)$. 

\end{enumerate}
\end{prop}
\begin{proof}
By Remark \ref {rem-var-Mobius} (2) and Proposition \ref{psl-equiv}, 
the conditions (1),(2) and (3) are all equivalent by the classical facts on M\"{o}bius transformation, see \cite[Section 13.1]{FM}. 

We prove $(3)\Rightarrow(4)$. 
The cohomological Fourier--Mukai transform $\Phi^H \in \SL(2, \bZ)$ fixes (up to sign) a primitive integral vector $v\in\bR^2$ 
(\cite[Section 13.1]{FM}), 
which is also a cohomology class of a spherical object. 
By Proposition \ref{prop:spherical}, we can take a $\mu$-stable sheaf $E$ satisfying $\ch(E)=v$. 
By composing with shifts if necessary, we may assume that $\Phi (E)$ is also a $\mu$-stable sheaf. 
Then since $\rk E>0$ and $\rk\Phi(E)>0$, 
we have $\ch(\Phi (E))=\ch(E)$. 
Set $\ch(E)=(r,d)$.
It is well-known that there exists an autoequivalence $F_{r,d}\in\Aut(D^b(X))$, 
which yields a one-to-one correspondence between the structure sheaf $\mcO_x$ of a point $x\in X$ and a $\mu$-stable sheaf $F_{r,d}(\mcO_x)$ of rank $r$ and degree $d$ (\cite[Proposition 3]{hp}). 
Let $f\in\Aut(X)$ be a translation satisfying 
\[
f^*(F_{r,d}^{-1}(E))=F_{r,d}^{-1}\Phi (E). 
\]
Since $f^*$ (composed with shifts) is an element of $\mcI(X)$, so is the conjugation $F_{r,d}\circ f^*\circ F_{r,d}^{-1}$. 
This autoequivalence is our desired $\Psi$. 

The assertion $(4)\Rightarrow(2)$ easily holds by the following argument. 
The condition (4) implies that 
$\Phi^H \in \SL(2, \bZ)$ fixes (up to sign) a vector $\ch(E)$, 
hence $|{\rm tr}\rho(\Phi)|=2$. 
\end{proof}

To state the characterization of hyperbolic autoequivalences, 
we recall the pseudo-Anosov property of autoequivalences. 

\begin{defin}[{\cite[Definition 4.1]{DHKK}, \cite[Definition 4.6]{Kik-curvature} and \cite[Definition 2.13]{FFHKL}}]
Let $\Phi\in\Aut(D^b(X))$ be an autoequivalence. 
\begin{enumerate}
\item
$\Phi$ is {\it pseudo-Anosov in the sense of \cite{DHKK}}
if
there exists a stability condition $\sigma_\Phi\in\Stab(X)$ and an element $\Tilde{g}_\Phi\in\widetilde{{\rm GL}}_+(2,\bR)$ such that 
\begin{enumerate}
\item
$g_\Phi=
\begin{pmatrix}
\frac{1}{r}&0\\
0&r
\end{pmatrix}
\text{ or }
\begin{pmatrix}
r&0\\
0&\frac{1}{r}
\end{pmatrix}
\in {\rm GL}_+(2,\bR)$ for some $|r|>1$. 
\item
$\Phi.\sigma_\Phi=\sigma_\Phi.\Tilde{g}_\Phi$
\end{enumerate}
where $g_\Phi\in{\rm GL}_+(2,\bR)$ is a natural projection of $\Tilde{g}_\Phi$ via $\widetilde{{\rm GL}}_+(2,\bR)\to{\rm GL}_+(2,\bR)$. 
We call $\lambda_\Phi:=|r|>1$ the {\it stretch-factor} of $\Phi$. 
\item
$\Phi$ is {\it pseudo-Anosov in the sense of \cite{FFHKL}}
if
there exists a stability condition $\sigma_\Phi\in\Stab(X)$ and $\lambda_\Phi>1$ such that for any non-zero object $E\in D^b(X)$, we have
\[
\displaystyle\limsup_{n\rightarrow\infty}\frac{1}{n}\log m_{\sigma}(\Phi^n E)
=\log\lambda_\Phi. 
\]
We call $\lambda_\Phi$ the {\it stretch-factor} of $\Phi$. 
\end{enumerate}
\end{defin}

Let 
\[
h_0(-): \Aut(D^b(X))\to\bR_{\ge0};~
\Phi\mapsto h_0(\Phi)
\]
be the categorical entropy (\cite[Definition 2.5]{DHKK}). 

\begin{prop}\label{prop-hyp}
Let $\Phi\in\Aut(D^b(X))$ be an autoequivalence which is non-trivial in $\Gamma(X)$. 
The following conditions are equivalent. 
\begin{enumerate}
\item
$\Phi$ is hyperbolic. 
\item
$|{\rm tr}\rho(\Phi)|>2$. 
\item
$\Phi$ has two fixed points in $\partial(\Stab(X)/\bC)$. 
\item
$h_0(\Phi)>0$
\item
$\Phi$ is pseudo-Anosov in the sense of \cite{DHKK}. 
\item
$\Phi$ is pseudo-Anosov in the sense of \cite{FFHKL}. 
\end{enumerate}
\end{prop}
\begin{proof}
By Remark \ref {rem-var-Mobius} (2) and Proposition \ref{psl-equiv}, 
the conditions (1),(2) and (3) are all equivalent by the classical facts on M\"{o}bius transformation, see \cite[Section 13.1]{FM}. 

The conditions (2) and (4) are equivalent by \cite[Theorem 3.1]{curve-entropy}. 
The conditions (2) and (5) are equivalent by \cite[Proposition 4.14]{Kik-curvature}. 
The conditions (5) and (6) are equivalent by \cite[Proposition 4.13, 4.14]{Kik-curvature} and \cite[Proposition 3.11]{FFHKL}. 
\end{proof}
Note that the categorical entropy of hyperbolic autoequivalences of elliptic curves
is equal to the translation length, the stretch factor, the mass-growth \cite[Lemma 4.15]{Kik-curvature} and the spectral radius \cite[Theorem 3.1]{curve-entropy}. 

\vspace{2mm}
The above three propositions (Proposition \ref{prop-ell}, \ref{prop-para} and \ref{prop-hyp}) and Corollary \ref{non-trivial} imply the following: 
\begin{thm}[Nielsen--Thurston classification]\label{thm-NT}
Each autoequivalence of $D^b(X)$
which acts on $\Stab(X)/\bC$ non-trivially 
is of exactly one of the following types: periodic, reducible, or pseudo-Anosov. 
\end{thm}


\section{The case of the projective line}
\label{sec:P1}
We first recall a standard fact on stability conditions on $\bP^1$. 
\begin{prop}[{\cite[Corollary 3.4]{Oka}}] \label{prop:P1heart}
A heart admitting a stability condition on $\Stab(\bP^1)$ 
is of the following form: 
\begin{equation}
\mcA_j:={\rm Coh}(\bP^1)[j]
\end{equation}
or
\begin{equation}
\mcA_{p,i,j}:=\langle \mcO_{\bP^1}(i-1)[p+j], \mcO_{\bP^1}(i)[j]\rangle_{ex}
\end{equation}
for some $i,j\in\bZ$ and $p\in\bZ_{>0}$. 
In the latter case, 
the only stable objects in the heart are 
$\mcO_{\bP^1}(i-1)[p+j]$ and $\mcO_{\bP^1}(i)[j]$, and 
if $p\geq 2$, any object in the heart $\mcA_{p,i,j}$ is 
of the form 
$\mcO_{\bP^1}(i-1)[p+j]^{\oplus k} \oplus \mcO_{\bP^1}(i)^{\oplus l}[j]$ 
for some $k, l \in \bZ_{\geq 0}$. 
\end{prop}

We will use the following easy lemma: 
\begin{lem}[{\cite[Lemma 3.2]{BMW}}] \label{lem:generateP1}
For any $n,k\in\bZ$, there exist exact triangles in $D^b(\bP^1)$: 
\begin{align}
    &\mcO_{\bP^1}(k+1)^{\oplus n-k} \to \mcO_{\bP^1}(n) \to \mcO_{\bP^1}(k)[1]^{\oplus n-k-1} 
    \quad \mbox{ if } k<n-1, \label{eq:exn>0}\\
    &\mcO_{\bP^1}(k+1)[-1]^{\oplus k-n} \to \mcO_{\bP^1}(n) \to \mcO_{\bP^1}(k)^{\oplus k-n+1} 
    \quad \mbox{ if } k>n, \label{eq:exn<0} \\
    &\mcO_{\bP^1}(k+1) \to \mcO_x \to \mcO_{\bP^1}(k)[1] \quad \mbox{ for } x \in \bP^1. 
    \label{eq:exOx}
\end{align}
Moreover, any exact triangle $A\to M\to B$ with $M$ either $\mcO_{\bP^1}(n)$ or $\mcO_x$ and $\Hom^{\le0}(A,B)=0$ is of one of the above forms.
\end{lem}
A proof of the above lemma is given by calculating factors of objects via the semi-orthogonal decomposition 
$D^b(\bP^1)=\langle \mcO_{\bP^1}(-1), \mcO_{\bP^1} \rangle$.


\subsection{Non-injectivity of $\bP m$}
We show that the map $\bP m$
fails to be injective in the case of the projective line. 
\begin{prop}\label{non-injectivity-p^1}
The map 
\[
\bP m \colon \Stab(\bP^1)/\bC \to \bP^{\mcS}_{\geq 0}
\]
is NOT injective 
for any choice of a set 
$\mcS \subset \Ob(D^b(\bP^1))$. 
\end{prop}
\begin{proof}
Let us consider the heart 
\[
\mcA_{p,i}:=\mcA_{p,i,0}
=\left\langle 
\mcO_{\bP^1}(i-1)[p], \mcO_{\bP^1}(i)
\right\rangle
\subset D^b(\bP^1), 
\]
and take elements $\zeta, \eta_j \in \bH$ 
($j=1, 2$) satisfying 
\[
|\eta_1|=|\eta_2|, \quad 
\arg(\eta_1) \neq \arg(\eta_2). 
\]
For $j=1, 2$, 
we define the central charge functions 
$Z_j \colon K(\mcA) \to \bC$ by the formula 
\[
Z_j\left(\mcO_{\bP^1}(i-1)[p] \right)
:=\zeta, \quad 
Z_j\left(\mcO_{\bP^1}(i)\right)
:=\eta_j. 
\]
Then $\sigma_j=(Z_j, \mcA)$ 
are stability conditions on $D^b(\bP^1)$. 
By Proposition \ref{prop:P1heart}, 
we conclude that 
\[
{m}_{\sigma_1}(E_{k, l})
=k|\zeta|+l|\eta_1|
=k|\zeta|+l|\eta_2|
={m}_{\sigma_2}(E_{k, l})
\]
for all $k, l \geq 0$, i.e., 
${\bP m}(\overline{\sigma}_1)
={\bP m}(\overline{\sigma}_2)$ 
for any choice of $\mcS$. 
\end{proof}

\begin{rmk}\label{rem:non-inj}
We can show that the above non-injectivity result holds for a general triangulated category with a full strong exceptional collection. 
\end{rmk}

\subsection{Explicit description of the image of $\bP m$}
\label{subsection-exp-im-p1}
In this subsection, 
we describe the image of $\bP m$ explicitly. 

Recall from Proposition \ref{prop:P1heart} that 
for any stability condition $\sigma \in \Stab(\bP^1)$, 
its heart is either $\mcA_j$ or $\mcA_{p, i, j}$ 
for some $p, i, j \in \bZ$ with $p >0$. 
We therefore decompose the space $\Stab(\bP^1)/\bC$ as follows: 
\begin{align*}
&\Stab(\bP^1)/\bC=\left( \Geo(\bP^1)/\bC \right) \bigsqcup 
\Big(\bigsqcup_{i \in \bZ} \mcH_i \Big), \\
&\mcH_i \coloneqq \left\{
\overline{\sigma} \in \Stab(\bP^1)/\bC 
~\middle|~
\begin{aligned}
&\sigma \mbox{ is not geometric}, \\
&\mbox{the heart of } \sigma \mbox{ is } 
\mcA_{p,i,0} \mbox{ for some } p >0
\end{aligned}
\right\}. 
\end{align*}

For a point $x \in\bP^1$, we set $\mcS_3:=\{\mcO_x,\mcO_{\bP^1}(-1),\mcO_{\bP^1}\}$ and consider the map
\[
\bP m \colon \Stab(\bP^1)/\bC \to \bP^{\mcS_3}_{\geq 0}. 
\]
Then by the proof of Proposition \ref{non-injectivity-p^1}, $\bP m$ is not injective. 
Recall that we have already determined 
$\bP m(\Geo(\bP^1)/\bC) \subset \bP^{\mcS_3}_{\geq0}$
in \S4.3. 
In the following, we describe the image of $\mcH_i$ for each $i \in \bZ$. 

We define a subset $\Delta_i\subset\bR\bP^2_{\ge0}$ for each $i\in\bZ$ as follows: 
\begin{eqnarray*}
\Delta_i&:=&
\{[1:X:Y]\in\bR\bP^2_{\ge0}\mid~Y=X-1,~i+1<x<i+2\}
\quad\mbox{ for }i\ge0\\
\Delta_{-1}&:=&
\{[1:X:Y]\in\bR\bP^2_{\ge0}\mid~Y=-X+1,~0<x<1\}\\
\Delta_i&:=&
\{[1:X:Y]\in\bR\bP^2_{\ge0}\mid~Y=X+1,~-i-2<x<-i-1\}
\quad\mbox{ for }i\le-2 
\end{eqnarray*}
For the subset $\Delta\subset\bR\bP^2_{\ge0}$ in \S4.3,
we clearly have
\[
\bigsqcup_{i\in\bZ}\Delta_i\subsetneq\overline{\Delta}\backslash\Delta. 
\]

Identifying $\bP^{\mcS_3}_{\ge0}$ with $\bR\bP^2_{\ge0}$,
we obtain the following result. 
\begin{thm}\label{thm-explicit-im-p^1}
The image $\bP m(\mcH_i)$ is equal to $\Delta_i$ for each $i\in\bZ$. 
\end{thm}

\begin{proof}
Take $\overline{\sigma} \in \mcH_i$. 
We have 
\[\bP m(\overline{\sigma})=\left[1: \frac{m_\sigma(\mcO_{\bP^1}(-1))}{m_\sigma(\mcO_x)}:\frac{m_\sigma(\mcO_{\bP^1})}{m_\sigma(\mcO_x)} \right].\]
Put 
\[s:=m_\sigma(\mcO_{\bP^1}(i)), t:=m_\sigma(\mcO_{\bP^1}(i-1)),\]
and
\[X:=\frac{m_\sigma(\mcO_{\bP^1}(-1))}{m_\sigma(\mcO_x)}, Y:=\frac{m_\sigma(\mcO_{\bP^1})}{m_\sigma(\mcO_x)}.\]
We describe $X$ and $Y$ in terms of $s$ and $t$.
By Proposition \ref{prop:P1heart}, we have 
\[ s=|Z(\mcO_{\bP^1}(i))|, t=|Z(\mcO_{\bP^1}(i-1))|.\]
Note that considering the action of $\sqrt{-1}\mathbb{R} \subset \mathbb{C}$, $s$ and $t$ can be arbitrary positive real numbers.

First, we assume that $i=0$. 
By Lemma \ref{lem:generateP1}, we have an exact triangle
\begin{equation}\label{filtration}
\mcO_{\bP^1} \to \mcO_x \to \mcO_{\bP^1}(-1)[1]. 
\end{equation}
If $\mcO_x$ is not $\sigma$-semistable, the exact triangle (\ref{filtration}) is the Harder-Narasimhan filtration of $\mcO_x$ with respect to $\sigma$. If $\mcO_x$ is $\sigma$-semistable, the exact triangle (\ref{filtration}) is the Jordan-H\"older filtration of $\mcO_x$ with respect to $\sigma$. Therefore, we obtain
\begin{align*}
    m_\sigma(\mcO_x) &=  |Z(\mcO_{\bP^1})|+|Z(\mcO_{\bP^1}(-1)[1]| \\
    &=s+t
\end{align*}
and 
\[X=\frac{t}{s+t}, \ Y=\frac{s}{s+t}. \]
So we have 
\[\bP m(\mcH_{0})=\Delta_{0}. \]


Next, assume that $i \geq 1$. By Lemma \ref{lem:generateP1}, we have exact triangles
\begin{align*}
    &\mcO_{\bP^1}(i)[-1]^{\oplus i} \to \mcO_{\bP^1}(-1) \to \mcO_{\bP^1}(i-1)^{\oplus i+1}, \\
    &\mcO_{\bP^1}(i)[-1]^{\oplus i-1} \to \mcO_{\bP^1} \to \mcO_{\bP^1}(i-1)^{\oplus i}, \\
    &\mcO_{\bP^1}(i) \to \mcO_x \to \mcO_{\bP^1}(i-1)[1]. 
\end{align*}

As in the case of $i=0$, the above exact triangles are the Harder--Narasimhan or the Jordan--H\"older filtrations with respect to $\sigma$. Therefore, we obtain

\begin{align*}
    m_\sigma(\mcO_x) &=  |Z(\mcO_{\bP^1}(i))|+|Z(\mcO_{\bP^1}(i-1)[1]| \\
    &=s+t,\\
    m_\sigma(\mcO_{\bP^1}(-1)) &=  i|Z(\mcO_{\bP^1}(i)[-1])|+(i+1)|Z(\mcO_{\bP^1}(i)| \\
    &=is+(i+1)t,\\
    m_\sigma(\mcO_{\bP^1}) &=  (i-1)|Z(\mcO_{\bP^1}(i))[-1]|+i|Z(\mcO_{\bP^1}(i-1)| \\
    &=(i-1)s+it
\end{align*}
and 
\[X=\frac{is+(i+1)t}{s+t}, \ Y=\frac{(i-1)s+it}{s+t}. \]
Hence, we have 
\[\bP m(\mcH_i)=\Delta_i\]
for $i\geq 1$.

Finally, assume that $i \leq -1$.
As in the case of $i \geq 0$, we obtain
\begin{align*}
    m_\sigma(\mcO_x) &=  |Z(\mcO_{\bP^1}(i))|+|Z(\mcO_{\bP^1}(i-1)[1]| \\
    &=s+t,\\
 m_\sigma(\mcO_{\bP^1}(-1)) &=  -i|Z(\mcO_{\bP^1}(i))|+(-i-1)|Z(\mcO_{\bP^1}(i-1)[1]| \\
 &=-is+(-i-1)t,\\
m_\sigma(\mcO_{\bP^1}) &=  (-i+1)|Z(\mcO_{\bP^1}(i))[-1]|-i|Z(\mcO_{\bP^1}(i-1)[1]| \\
&=(-i+1)is-it
\end{align*}
and
\[X=\frac{-is+(-i-1)t}{s+t}, \ Y=\frac{(-i+1)s-it}{s+t}. \]
Hence, we have
\[m_\sigma(\mcH_i)=\Delta_i \]
for $i\leq-1$.
\end{proof}

Therefore, we see that 
the image $\bP m(\Stab(\bP^1)/\bC)$ gives a partial compactification of $\Geo(\bP^1)/\bC$. 

\begin{figure}[htb]
\begin{center}
\begin{tikzpicture}[scale=1]
\draw [->] (-1, 0) -- (3, 0) node[right]{$X$}; 
\draw[->] (0, -1) -- (0, 4) node[right]{$Y$}; 
\draw[dashed, domain=-1:2] plot(\x, -\x+1); 
\draw[thick, domain=0:1] plot(\x, -\x+1); 
\draw[dashed, domain=0:3] plot(\x, \x-1); 
\draw[thick, domain=1:2] plot(\x, \x-1); 
\draw[thick, domain=2:3] plot(\x, \x-1) node[right]{$Y=X-1$}; 
\draw[dashed, domain=-1:3] plot(\x, \x+1) node[right]{$Y=X+1$}; 
\draw[thick, domain=0:1] plot(\x, \x+1); 
\draw[thick, domain=1:2] plot(\x, \x+1); 
\draw[thick, domain=2:3] plot(\x, \x+1); 
\fill[lightgray] (1, 0)--(0, 1)--(3, 4)--(3, 2)--cycle;
\fill[fill=white] (0, 1) circle (0.1cm); 
\draw (0, 1) circle (0.11cm); 
\fill[fill=white] (1, 2) circle (0.1cm); 
\draw (1, 2) circle (0.11cm); 
\fill[fill=white] (1, 0) circle (0.1cm); 
\draw (1, 0) circle (0.11cm); 
\fill[fill=white] (2, 1) circle (0.1cm); 
\draw (2, 1) circle (0.11cm); 
\fill[fill=white] (2, 3) circle (0.1cm); 
\draw (2, 3) circle (0.11cm); 
\draw (-1.3, 2) node[above]{$Y=-X+1$}; 
\end{tikzpicture}
\end{center}
\caption{The image $\bP m(\Stab(\bP^1)/\bC)$ in $\bR\bP^2_{\geq 0}$.}
\end{figure}

\bibliographystyle{alpha}
\bibliography{maths}

\end{document}

%% file: packages.tex
\theoremstyle{plain}
\newtheorem{thm}{Theorem}[section]
\newtheorem{prop}[thm]{Proposition}
\newtheorem{def-prop}[thm]{Definition-Proposition}
\newtheorem{lem}[thm]{Lemma}
\newtheorem{cor}[thm]{Corollary}
\newtheorem*{thm*}{Theorem}

\theoremstyle{definition}
\newtheorem{defin}[thm]{Definition}

\newtheorem*{NaC}{Notation and Convention}
\newtheorem*{ACK}{Acknowledgement}

\theoremstyle{remark}
\newtheorem{rmk}[thm]{Remark}

\newtheorem{ques}[thm]{Question}
\newtheorem{ex}[thm]{Example}

\newtheorem{assump}[thm]{Assumption}

\DeclareMathOperator{\ch}{ch}

\DeclareMathOperator{\rk}{rk}

\newcommand{\dR}{\mathbf{R}}

\newcommand{\bP}{\mathbb{P}}
\newcommand{\bC}{\mathbb{C}}

\newcommand{\bR}{\mathbb{R}}

\newcommand{\bZ}{\mathbb{Z}}
\newcommand{\bH}{\mathbb{H}}
\newcommand{\mcA}{\mathcal{A}}

\newcommand{\mcD}{\mathcal{D}}

\newcommand{\mcH}{\mathcal{H}}
\newcommand{\mcI}{\mathcal{I}}

\newcommand{\mcO}{\mathcal{O}}
\newcommand{\mcP}{\mathcal{P}}

\newcommand{\mcS}{\mathcal{S}}
\newcommand{\mcT}{\mathcal{T}}
\newcommand{\mcU}{\mathcal{U}}

\DeclareMathOperator{\Hom}{Hom}

\DeclareMathOperator{\Coh}{Coh}

\DeclareMathOperator{\Pic}{Pic}

\DeclareMathOperator{\ext}{ext}
\DeclareMathOperator{\Ext}{Ext}

\DeclareMathOperator{\Stab}{Stab}
\DeclareMathOperator{\Geo}{Geo}
\DeclareMathOperator{\cl}{cl}

\DeclareMathOperator{\SL}{SL}

\DeclareMathOperator{\ST}{ST}

\DeclareMathOperator{\Ob}{Ob}

\DeclareMathOperator{\Sph}{Sph}

\DeclareMathOperator{\Aut}{Aut}

\DeclareMathOperator{\RHom}{RHom}

\DeclareMathOperator{\MCG}{MCG}

\DeclareMathOperator{\Fuk}{Fuk}

\newcommand{\iso}{\xrightarrow{\sim}}

%% file: main.bbl
\newcommand{\etalchar}[1]{$^{#1}$}
\begin{thebibliography}{DHKK14}

\bibitem[BBL22]{BBL}
A.~Bapat, L.~Becker, and A.~M. Licata.
\newblock {$q$-deformed rational numbers and the 2-Calabi--Yau category of type
  $A_2$}.
\newblock {\em arXiv:2202.07613}, 2022.

\bibitem[BDL20]{bdl20}
A.~Bapat, A.~Deopurkar, and A.~M. Licata.
\newblock A {T}hurston compactification of the space of stability conditions.
\newblock {\em arXiv:2011.07908}, 2020.

\bibitem[BMW15]{BMW}
A.~Bertram, S.~Marcus, and J.~Wang.
\newblock The stability manifolds of {$\bP^1$} and local {$\bP^1$}.
\newblock {\em Contemporary Mathematics}, 647:1--17, 2015.

\bibitem[Bri06]{Bri-Teich}
T.~Bridgeland.
\newblock Spaces of stability conditions.
\newblock {\em arXiv:math/0611510}, 2006.

\bibitem[Bri07]{bri}
T.~Bridgeland.
\newblock Stability conditions on triangulated categories.
\newblock {\em Ann. of Math. (2)}, 166(2):317--345, 2007.

\bibitem[BS15]{BS}
T.~Bridgeland and I.~Smith.
\newblock Quadratic differentials as stability conditions.
\newblock {\em Publ. Math. IHES}, 121:155--278, 2015.

\bibitem[DHKK14]{DHKK}
G.~Dimitrov, F.~Haiden, L.~Katzarkov, and M.~Kontsevich.
\newblock Dynamical systems and categories.
\newblock {\em Contemporary Mathematics}, 621:133--170, 2014.

\bibitem[FFH{\etalchar{+}}19]{FFHKL}
Y.-W. Fan, S.~Filip, F.~Haiden, L.~Katzarkov, and Y.~Liu.
\newblock {On pseudo-Anosov autoequivalences}.
\newblock {\em arXiv:1910.12350}, 2019.

\bibitem[FLP12]{FLPV}
A.~Fathi, F.~Laudenbach, and V.~Po\'{e}naru.
\newblock {\em Thurston's work on surfaces}, volume~48 of {\em Mathematical
  Notes}.
\newblock Princeton University Press, Princeton, NJ, 2012.
\newblock Translated from the 1979 French original by Djun M. Kim and Dan
  Margalit.

\bibitem[FM12]{FM}
B.~Farb and D.~Margalit.
\newblock {\em A primer on mapping class groups}.
\newblock volume 49 of Princeton Mathematical Series. Princeton University
  Press, Princeton, 2012.

\bibitem[GMN13]{GMN}
D.~Gaiotto, G.~W. Moore, and A.~Neitzke.
\newblock {Wall-crossing, Hitchin systems, and the WKB approximation}.
\newblock {\em Advances in Mathematics}, 234:239--403, 2013.

\bibitem[HP05]{hp}
G.~Hein and D.~Ploog.
\newblock Fourier-{M}ukai transforms and stable bundles on elliptic curves.
\newblock {\em Beitr\"{a}ge Algebra Geom.}, 46(2):423--434, 2005.

\bibitem[Ike21]{Ike}
A.~Ikeda.
\newblock Mass growth of objects and categorical entropy.
\newblock {\em Nagoya Mathematical Journal}, 244:136--157, 2021.

\bibitem[Kik17]{curve-entropy}
K.~Kikuta.
\newblock On entropy for autoequivalences of the derived category of curves.
\newblock {\em Advances in Mathematics}, 308:699--712, 2017.

\bibitem[Kik22]{Kik-curvature}
K.~Kikuta.
\newblock Curvature of the space of stability conditions.
\newblock {\em manuscripta mathematica}, online, 2022.

\bibitem[KKO]{KKOk3}
K.~Kikuta, N.~Koseki, and G.~Ouchi.
\newblock {T}hurston compactifications of spaces of stability conditions on
  {K3} surfaces.
\newblock {\em in preparation}.

\bibitem[Mac07]{Mac}
E.~Macr\`i.
\newblock Stability conditions on curves.
\newblock {\em Math. Res. Lett.}, 14(4):657–672, 2007.

\bibitem[Muk81]{muk81}
S.~Mukai.
\newblock Duality between {$D(X)$} and {$D(\hat X)$} with its application to
  {P}icard sheaves.
\newblock {\em Nagoya Math. J.}, 81:153--175, 1981.

\bibitem[Oka06]{Oka}
S.~Okada.
\newblock Stability manifold of {$\bP^1$}.
\newblock {\em J. Algebraic Geom.}, 15:487--505, 2006.

\bibitem[PZ98]{pz98}
A.~Polishchuk and E.~Zaslow.
\newblock Categorical mirror symmetry: the elliptic curve.
\newblock {\em Adv. Theor. Math. Phys.}, 2(2):443--470, 1998.

\bibitem[Woo12]{Woo}
J.~Woolf.
\newblock Some metric properties of spaces of stability conditions.
\newblock {\em Bull. Lond. Math. Soc.}, 44(6):1274–1284, 2012.

\end{thebibliography}
